\documentclass[10pt, reqno]{amsart}
\usepackage{amsfonts}
\usepackage{amsbsy}
\usepackage{tikz}
\usepackage{amsmath}
\usepackage{amssymb}
\usepackage{amsthm}
\usepackage{syntonly}
\usepackage{url} %to show address of website
\usepackage{amscd} %for diagram
\usepackage{tikz-cd} % for diagram
\usepackage{bm} %for bold-faced letter
\usepackage{mathdots} % for typing various dots

\usepackage[colorlinks,linkcolor=orange, anchorcolor=green, citecolor=purple ]{hyperref}
\usepackage{graphicx} %for larger math equations
\usepackage{mathrsfs} % for script letter
\usepackage{comment} % this is to use \comment demand 
\usepackage{bbm} % this is to type script numbers

\usepackage[utf8x]{inputenc}
    \usepackage{ucs}
\newcommand\quotient[2]{
        \mathchoice
            {% \displaystyle
                \text{\raise1ex\hbox{$\#1$}\Big/\lower1ex\hbox{$\#2$}}%
            }
            {% \textstyle
                \#1\,/\,\#2
            }
            {% \scriptstyle
                \#1\,/\,\#2
            }
            {% \scriptscriptstyle  
                \#1\,/\,\#2
            }
    }% to do quotient

\newcommand{\R}{\mathbb{R}} %shortcut for math letter
\newcommand{\Z}{\mathbb{Z}}
\newcommand{\Q}{\mathbb{Q}}

\newcommand{\bma}{\bm{\mathrm{a}}}
\newcommand{\bmA}{\bm{\mathrm{A}}}

\newcommand{\bmD}{\bm{\mathrm{D}}}

\newcommand{\bmF}{\bm{\mathrm{F}}}

\newcommand{\bmK}{\bm{\mathrm{K}}}
\newcommand{\bmL}{\bm{\mathrm{L}}}
\newcommand{\bmM}{\bm{\mathrm{M}}}

\newcommand{\bmP}{\bm{\mathrm{P}}}

\newcommand{\bmU}{\bm{\mathrm{U}}}
\newcommand{\bmv}{\bm{\mathrm{v}}}

\newcommand{\bmy}{\bm{\mathrm{y}}}

\newcommand{\rmA}{\mathrm{A}}

\newcommand{\rmF}{\mathrm{F}}
\newcommand{\rmG}{\mathrm{G}}

\newcommand{\rmK}{\mathrm{K}}

\newcommand{\rmm}{\mathrm{m}}
\newcommand{\rmM}{\mathrm{M}}
\newcommand{\rmP}{\mathrm{P}}

\newcommand{\rmU}{\mathrm{U}}

\newcommand{\fraka}{\mathfrak{a}}

\newcommand{\frakg}{\mathfrak{g}}

\newcommand{\frakk}{\mathfrak{k}}

\newcommand{\frakm}{\mathfrak{m}}
\newcommand{\frakp}{\mathfrak{p}}

\newcommand{\fraku}{\mathfrak{u}}

\newcommand{\scrB}{\mathscr{B}}

\newcommand{\scrI}{\mathscr{I}}

\newcommand{\scrU}{\mathscr{U}}

\newcommand{\calC}{\mathcal{C}}

\newcommand{\calU}{\mathcal{U}}
\newcommand{\calW}{\mathcal{W}}

\newcommand{\whmu}{\widehat{\mu}}

\newcommand{\wtB}{\widetilde{B}}

\newcommand{\wtX}{\widetilde{X}}

\newcommand{\wtomega}{\widetilde{\omega}}

\newcommand{\ep}{\varepsilon}

\newtheorem{thm}{Theorem}[section]

\newtheorem{coro}[thm]{Corollary}

\newtheorem{lem}[thm]{Lemma}
\newtheorem{prop}[thm]{Proposition}

\newcommand{\bs}{\backslash}  %shortcut for other command
\newcommand{\la}{\langle}
\newcommand{\ra}{\rangle}

\DeclareMathOperator{\SL}{SL}

\DeclareMathOperator{\SO}{SO}

\newcommand{\fraksl}{\mathfrak{sl}}

\DeclareMathOperator{\diff}{d}

\DeclareMathOperator{\diffbmy}{d\textbf{y}}

\DeclareMathOperator{\Ad}{Ad}

\DeclareMathOperator{\diag}{diag}
\DeclareMathOperator{\dist}{dist}

\DeclareMathOperator{\height}{ht}
\DeclareMathOperator{\hor}{hor}

\DeclareMathOperator{\Leb}{Leb}

\DeclareMathOperator{\new}{new}
\DeclareMathOperator{\old}{old}
\DeclareMathOperator{\ONB}{ONB}

\DeclareMathOperator{\supp}{supp}

\DeclareMathOperator{\Tr}{Tr}
\DeclareMathOperator{\Vol}{Vol}

\newcommand{\norm}[1]{\left\lVert#1\right\rVert}

\begin{document}
\author{Runlin Zhang}
\thanks
{Current address: Beijing International Center for Mathematical Research, Peking University,
1000871 Beijing, China}

% It is required to enter 2020 MSC.

% Please provide minimum  5 keywords.

% Email address of each of all authors is required.
% You may list email addresses of all other authors, separately.
 \email{zhangrunlinmath@outlook.com}

\title[Count horocycles]{Count lifts of non-maximal closed horocycles on $SL_N(\mathbb{Z}) \backslash SL_N(\mathbb{R})/SO_N({\mathbb{R}})$}

\maketitle

\begin{abstract}
A closed horocycle $\mathcal{U}$ on  $SL_N(\mathbb{Z}) \backslash SL_N(\mathbb{R})/SO_N({\mathbb{R}})$ has many lifts to the universal cover $SL_N(\mathbb{R})/SO_N({\mathbb{R}})$. Under some conditions on the horocycle, we give a precise asymptotic count of its lifts of bounded distance away from a given base point in the universal cover. This partially generalizes previous work of Mohammadi--Golsefidy.
\end{abstract}

\tableofcontents

\section{Introduction}

Let $\wtX$ be a global Riemannian symmetric space of noncompact type and let $X$ be a locally symmetric space that is a quotient of $\wtX$. Let $\pi: \wtX \to X$ denote the covering map. For a point $x \in X$, $\pi^{-1}(x)$ is a discrete set in $\wtX$. For each $R>0$, consider the intersection of  $\pi^{-1}(x)$ with a ball of radius $R$ centered at some fixed point in $\wtX$. Then it becomes a finite set and an asymptotic count as $R \to +\infty$ has been found in \cite{EskMcM93} under some irreducibility condition.

Now let $Y$ be a closed submanifold (or suborbifold) of $X$. Then $\pi^{-1}(Y)$ forms a locally finite family of lifts of $Y$. One may also ask asymptotically how many lifts of $Y$ there are.

In the present paper, we specialize to the situation where the locally symmetric space takes the form $X=\Gamma \bs \SL_N(\R)/ \SO_N(\R)$ with $\Gamma$ commensurable with $\SL_N(\Z)$ and
the closed submanifold takes the form $Y=\Gamma \bs\Gamma  U \SO_N(\R)/ \SO_N(\R)$ for some rational subgroup $U$. Then $\wtX=\SL_N(\R)/ \SO_N(\R)$ and each lift of $Y$ is a \textit{horocycle} in $\wtX$. Horocycles are interesting from the point view of geometry (see \cite[2.10]{KapoLeebPor17}) and topology (see \cite[12.3]{Schwer10}). When the horocycle is a horosphere, such a problem has been studied in \cite{GolMoha14}.

Two examples of our main theorem are presented below and the more general statement comes later as Theorem \ref{thmCount}.  
Let $o\in \SL_N(\R)/\SO_N(\R)$ be the identity coset. All distances below are induced from the standard trace form 
\begin{equation*}
    (X,Y)\mapsto \Tr(XY^{tr})
\end{equation*}
on the Lie algebra. To save notation, write $\rmK:=\SO_N(\R)$.

\subsection*{Example 1}\label{Example1}
Let $N=3$ and 
\begin{equation*}
 U:=
   \left\{ \left[ \begin{array}{ccc}
       1  &  *_1 &*_2\\
         & 1 & *_3\\
         &&  1
    \end{array}\right],   \; *_i \in \R, \;i=1,2,3 \right\} 
\end{equation*}
Then as $R\to +\infty$ 
\begin{equation*}
\begin{aligned}
          \#
     \left\{ \gamma U\rmK/\rmK \;\middle\vert\;
     \dist( \gamma U\rmK/\rmK , o)\leq R ,\;\gamma \in \SL_N(\Z)
     \right\} 
     \sim 
     R^{\frac{1}{2}}e^{2\sqrt{2}R}\cdot \frac{\pi^{\frac{1}{2}}\cdot3\cdot 2^{\frac{1}{4}}}
     {7\cdot \xi(2)\xi(3)}
\end{aligned}
\end{equation*}
where $\xi(s)$ is the completed Riemann zeta function. The notation $\sim$ means the ratio between the left hand side and the right hand side goes to $1$ as $R\to +\infty$.

This example is also a special case of \cite[Theorem 3]{GolMoha14} with a different constant. We suspect that this is due to firstly in the middle of page 1328 of loc. cit., the authors stated ``the stabilizer of $\calU$ is U'', which differs from Lemma \ref{lemGhoroDescription} by some connected components. And secondly,
at the last row, where there is a curly arrow, of the multi-lined equations 
on page 1331 of loc. cit.,  the Haar measures seem to be different from ours.
The next example is new.

\subsection*{Example 2}\label{Example2}
Let $N=3$ and 
\begin{equation*}
 U:=
   \left\{ \left[ \begin{array}{ccc}
       1  &  0 &*_1\\
         & 1 & *_2\\
         &&  1
    \end{array}\right],   \; *_i \in \R, \;i=1,2 \right\} 
\end{equation*}
Then as $R\to +\infty$ 
\begin{equation*}
      \#
     \left\{ \gamma U\rmK/\rmK \;\middle\vert\;
     \dist( \gamma U\rmK/\rmK , o)\leq R ,\;\gamma \in \SL_N(\Z)
     \right\} \\
     \sim  R^{\frac{1}{2}}e^{2\sqrt{2}R}\cdot
     \frac{\pi^{\frac{3}{2}}
     }{
     2^{\frac{1}{4}}\cdot \xi(2)\xi(3)
     }.
\end{equation*}

Our counting result Theorem \ref{thmCount} has some limitations.  It only treats certain special closed horocycles. This is because of the corresponding restriction put on the dynamical Theorem \ref{thmMain}. 
The proof of this uses the equidistribution theorem of \cite{EskMozSha96} and the non-divergence criterion in \cite{ZhangArxivBorelSerre}, both of which require the homogeneous measure to be ``defined over $\Q$''. But as the homogeneous measure associated with a horocycle has a compact part, Borel density lemma does not apply to show, and it is not true that, this is always defined over $\Q$. 
That being said, this should only be a technical matter.
Indeed, the work of \cite{EskMozSha96} essentially relies on the unipotent measure rigidity theorem of Ratner \cite{Rat91} and the linearization technique of \cite{DanMar91}. The work of  \cite{ZhangArxivBorelSerre} relies on \cite{KleMar98} (and \cite{DawGoroUllLi19} beyond the case of $\SL_N$). None of which (except for \cite{DawGoroUllLi19}) requires the homogeneous measure to be ``defined over $\Q$''. Thus it is possible to argue without this assumption. We hope to return to the general case in a future work. 
Also, the effective aspects of the counting problem are not treated here. The reader is referred to \cite{DabbsKellyLi16} for discussions.

Our main dynamical result Theorem \ref{thmMain} is stated in Section \ref{secMain} and the proof is given in the section after that. Our main counting result Theorem \ref{thmCount} is in Section \ref{secCount} where a large portion of work is devoted to the volume computations. Justifications of the examples above are given after Corollary \ref{coroCountSLZ}.

\section{Preliminaries and statement of the main results}\label{secMain}

\subsection{Preparations}
Let $e_1,...,e_N$ be the standard basis of $\R^N$ and let  $\norm{\cdot}$ denote the standard Euclidean norm. For $I \subset \{1,...,N\}$, write $e_I:=\wedge_{i\in I} e_i$, defined up to $\pm1$. Let $\R^I$ be the subspace of $\R^N$ defined by 
\begin{equation*}
    (x_1,...,x_N) \in \R^I \iff x_i = 0, \;\forall i \notin I.
 \end{equation*}
Let $\scrI_0:=\{I_1,...,I_{k_0}\}$ be a partition of $\{1,...,N\}$, namely, 
\begin{equation*}
    \{1,...,N\} = I_1 \sqcup I_2 \sqcup...\sqcup I_{k_0}.
\end{equation*}
Assume that for every pair $k_1<k_2$ and $i_1 \in I_{k_1}$, $i_2 \in I_{k_2}$, we have $i_1 < i_2$. An equivalence relation on $\{1,...,N\}$ is defined by 
\begin{equation*}
    i \sim_{\scrI_0} j \iff i,j\in I_k,\;\exists \,k\in\{1,...,k_0\}.
\end{equation*}
Let 
\begin{equation*}
\begin{aligned}
    \bmP_{\scrI_0}:=&
    \left\{
    g \in \SL_N\;\middle\vert\;
    g \R^{I_k} \subset \R^{\cup_{i\leq k}I_i},\;\forall k
    \right\} \\
    \bmU_{\scrI_0}:=&
    \left\{
    g \in \SL_N\;\middle\vert\;
    (id- g) \R^{I_k} \subset \R^{\cup_{i<k}I_i},\;\forall k
    \right\} \\
    \bmL_{\scrI_0}:=&
    \left\{
    g \in \SL_N\;\middle\vert\;
    g \R^{I_k} \subset \R^{I_k},\;\forall k
    \right\}  \\
    \bmM_{\scrI_0}:=&
    \left\{
    g \in \bmL_{\scrI_0} \;\middle\vert\;
    \det(g\vert_{\R^{I_k}})=1,\;\forall k
    \right\} \\
    \bmA_{\scrI_0}:=&
    \left\{
     g \in \bmL_{\scrI_0} \;\middle\vert\;
    g\vert_{\R^{I_k}}=d_k \,id,\;\forall k,\;\exists d_k>0
    \right\}\\
    \bmK_{\scrI_0}:=&
    \bmL_{\scrI_0} \cap \SO_N = \bmP_{\scrI_0} \cap \SO_N
\end{aligned}
\end{equation*}
where by convention $\R^{\cup_{i<1}I_i}:= \{0\}$ and $id$ denotes the identity matrix of suitable size.
As $\scrI_0$ varies, $\{\bmP_{\scrI_0}\}$  is a set of standard parabolic subgroups. Also, $\bmU_{\scrI_0}$ is the unipotent radical of $\bmP_{\scrI_0}$ and $\bmL_{\scrI_0}$ is the unique Levi subgroup of $\bmP_{\scrI_0}$ stable under transpose inverse. 
One also sees that $\bmM_{\scrI_0}= \oplus_{k=1,...,k_0} \bmM_{I_k}$, a direct sum of $\SL_{|I_k|}$'s, and $(\bmK_{\scrI_0})^{\circ}= \oplus_{k=1,...,k_0} \bmK_{I_k}$, a direct sum of $\SO_{|I_k|}$'s, where
\begin{equation*}
\begin{aligned}
         \bmM_{I_k} :=&\left\{
     g \in \bmM_{\scrI_0} \;\middle\vert\;
    g\vert_{\R^{I_j}}=id,\;\forall j\neq k
     \right\},\\
     \bmK_{I_k} :=&\left\{
     g \in \bmK_{\scrI_0} \;\middle\vert\;
    g\vert_{\R^{I_j}}=id,\;\forall j\neq k
     \right\}.\\
\end{aligned}
\end{equation*}
We shall use the Roman letter for the analytic identity components of the real points of these algebraic groups. For instance,
\begin{equation*}
\begin{aligned}
    \rmA_{\scrI_0}=&\left\{
    \bma=\diag(a_1,...,a_N)\;\middle\vert\;
    a_i=a_j,\;\forall i\sim_{\scrI_0}j;\;a_k>0,\;\forall k
    \right\} \\
    \rmK_{\scrI_0}=&\left\{
    g\in \SL_N(\R)\;\middle\vert\;
    g \vert_{\R^I_k} \in \SO_{|I_k|}(\R),\;\forall k
    \right\}.
\end{aligned}
\end{equation*}

Let $\Gamma$ be a lattice in $\SL_N(\R)$ commensurable with $\SL_N(\Z)$.
There are two types of  measures attached to $\scrI_0$:
\begin{itemize}
    \item[Type 1:] the unique probability Haar measure supported on $\rmM_{\scrI_0}\rmU_{\scrI_0}\Gamma/\Gamma$, call it $\mu_{\scrI_0}$;
    \item[Type 2:] the unique probability Haar measure supported on $\rmK_{\scrI_0}\rmU_{\scrI_0}\Gamma/\Gamma$, call it $\nu_{\scrI_0}$.
\end{itemize}
Translates of measures of Type 1 is the subject of \cite{GolMoha14}, where they are called ``horospherical''.  Their support in $\Gamma \bs \SL_N(\R)/\SO_N(\R)$ coincides with the quotient of some horosphere defined by level sets of Busemann functions (see \cite[1.10.1]{Eberleinbook}) only when $\scrI_0$ consists only two parts, namely when $\rmP_{\scrI_0}$ is a proper maximal parabolic.   These are most relevant for the study of counting rational points on flag varieties. 
We call measures of Type 2 \textbf{horocyclic} which is the subject of the current paper. These measures are supported on closed horocycles (see \cite[2.10]{KapoLeebPor17} for a geometric definition in the universal cover). As Type 1 and 2 has a nontrivial intersection when  $\scrI_0$ is most refined, i.e. $|I_k|=1$ for all $k$, the count of \textit{maximal} horocycles is also carried out in \cite{GolMoha14}. The novelty of this paper is the treatment of non-maximal cases.

As usual, given a sequence $(g_n)$ in $\SL_N(\R)$, the study of limiting measures of $(g_n \mu_{\scrI_0})$ (or $(g_n \nu_{\scrI_0})$) consists of two steps:
\begin{itemize}
    \item[Step 1.]  determine when the sequence of the measures is non-divergent;
    \item[Step 2.]  determine the limiting homogeneous measure in the non-divergent case.
\end{itemize}
For the first step we use the nondivergence criterion in \cite{ZhangArxivBorelSerre} (note that the group here may not be generated by $\R$-unipotents, so the work of \cite{DanMar91} is not sufficient) whereas for the second step we use \cite{EskMozSha96}. 

\subsection{Statement of the main results}
One has a natural bijection
\begin{equation*}
    \SL_N(\R) = \SO_N(\R)\rmM_{\scrI_0} \times \rmA_{\scrI_0} \times \rmU_{\scrI_0}
    =\SO_N(\R)(\oplus_{k} \rmM_{I_k}) \times \rmA_{\scrI_0} \times \rmU_{\scrI_0}.
\end{equation*}
We may further decompose 
\begin{equation*}
    \oplus_{k}\rmM_{I_k}= \oplus_{k} \rmK_{I_k}\rmA^+_{\rmM_{I_k}}\rmK_{I_k},
\end{equation*}
where 
\begin{equation*}
\begin{aligned}
    \rmA_{\rmM_{I_k}}:=& \left\{
     \bma=\diag(a_1,...,a_N)\in \SL_N(\R) \;\middle\vert\;
     a_{i}=1,\; \forall i \notin I_k;\;
     a_{i}>0,\;\forall i \in I_k
    \right\},
     \\
        \rmA^+_{\rmM_{I_k}}:=&\left\{
      \bma=\diag(a_1,...,a_N) \in \rmA_{\rmM_{I_k}}\;\middle\vert\;
     a_{i+1}/a_{i}\geq 1,\; \forall i,i+1 \in I_k
    \right\}.
\end{aligned}
\end{equation*}
Now we have a surjection
\begin{equation*}
    \SO_N(\R) \times  \oplus_{k} \rmA^+_{\rmM_{I_k}} \times (\oplus_{k} \rmK_{I_k})\times \rmA_{\scrI_0} \times \rmU_{\scrI_0} \to \SL_N(\R).
\end{equation*}
Thus we may write 
\begin{equation*}
    g_n = k_n \cdot \oplus a^k_n \cdot c_n \cdot b_n \cdot u_n.
\end{equation*}
And hence 
\begin{equation*}
    g_n \cdot \nu_{\scrI_0} = k_n \oplus a_n^k b_n \cdot \nu_{\scrI_0}.
\end{equation*}
As $(k_n)$ is a bounded sequence, it suffices to understand the limit behaviour of $(  \oplus a_n^k b_n \cdot \nu_{\scrI_0})$. 
By passing to a subsequence and modifying from left by a bounded sequence we may assume that the sequence $( \oplus a_n^k b_n)$ is \textbf{clean} in the following sense.
First $(\oplus a_n^k)$ is clean:
\begin{itemize}
    \item[1.] for each $k\in\{1,...,k_0\}$, either $(a_n^k)$ is unbounded or $a^k_n\equiv id $ for all $n$;
\end{itemize}
in which case we let 
\begin{equation*}
\begin{aligned}
     \scrI(a_n,\infty):=& \left\{ I_k \;\middle\vert\;
    (a^k_n) \text{ is unbounded } \right\}, \\
    \scrI(a_n,1):=& \left\{ I_k\;\middle\vert\;
    a^k_n\equiv id \right\}.
\end{aligned}
\end{equation*}
Second $(b_n)$ is clean:
\begin{itemize}
    \item[2.] for each $k\in\{1,...,k_0\}$, 
    $ \lambda_{I_1\cup...\cup I_k}(b_n) $ diverges to $+\infty$,  remains constantly equal to $1$, or converges to $0$.
\end{itemize}
where for a diagonal matrix $ \bma=\diag(a_1,...,a_N)$ and $I\subset\{1,...,N\}$, we denote $\lambda_I(a):=\prod_{i\in I} a_i$.
In this case define 
\begin{equation*}
\begin{aligned}
    \scrI(b_n,0):=& \left\{ I_k \;\middle\vert\;
    \lambda_{I_1\cup...\cup I_k}(b_n) \text{ converges to } 0 \right\} \\
        \{k_1<k_2<...<k_{l_0}=k_0\}:=&
    \{k=1,...,k_0\,\vert\,\lambda_{I_1\cup...\cup I_k}(b_n) \text{ equals to } 1 \}
\end{aligned}
\end{equation*}
and
\begin{equation*}
    \scrI_1(b_n):=\left\{
    I_1\cup...\cup I_{k_1}, I_{k_1+1}\cup...\cup I_{k_2},...,I_{k_{l_0-1}+1}\cup...\cup I_{k_{l_0}}
    \right\},
\end{equation*}
which is a partition coarser than $\scrI_0$.
From the definition, $(b_n)$ is contained in $\rmM_{\scrI_1(b_n)}$.

For simplicity define 
\begin{equation*}
\begin{aligned}
  &\scrI_1(\new):= \scrI_1(b_n)\setminus \scrI_0,\;
  \scrI_1(\old,\infty):= \scrI_1(b_n)\cap \scrI_0\cap \scrI(a_n,\infty),\\
  &\scrI_1(\infty):= \scrI_1(\new)\sqcup \scrI_1(\old,\infty),
  \\
  &\scrI_1(0):= \scrI_1(b_n)\cap \scrI_0\cap \scrI(a_n,0).
\end{aligned}
\end{equation*}

Thus for a clean sequence,
\begin{equation*}
    \scrI_1:=\scrI_1(b_n)=
    \scrI_1(\new)\sqcup\scrI_1(\old,\infty)\sqcup\scrI_1(0)= \scrI_1(\infty)\sqcup \scrI_1(0).
\end{equation*}

Now we have all the necessary terminologies to state our theorem.

\begin{thm}\label{thmMain}
Assume $(\oplus a_n^k b_n)$ is clean. Then 
\begin{itemize}
    \item [1.] $ (\oplus a_n^k b_n \cdot \nu_{\scrI_0})$ is non-divergent iff $\scrI(b_n,0)=\emptyset$;
    \item[2.] if $\scrI(b_n,0)=\emptyset$, then $(\oplus a_n^k b_n \cdot \nu_{\scrI_0})$ converges to the unique probability Haar measure supported on 
    \begin{equation*}
        (\oplus_{I\in \scrI_1(0)} \rmK_{I}) 
        (\oplus_{I\in \scrI_1(\infty)} \rmM_{I}) \rmU_{\scrI_1}\Gamma/\Gamma.
    \end{equation*}
\end{itemize}
\end{thm}
A sequence $(\mu_n)$ of probability measures is said to be \textit{non-divergent} iff for every $\ep>0$ there exists a bounded set $B$ such that $\mu_n(B)>1-\ep$ for all $n$.

By comparison, it follows from \cite[Theorem 1]{GolMoha14} that (in this case $(\oplus a^k_n)$ stabilizes the measure and plays no role)
\begin{thm}
Assume $(b_n)$ is clean. Then 
\begin{itemize}
    \item [1.] $ (\oplus b_n \cdot \mu_{\scrI_0})$ is non-divergent iff $\scrI(b_n,0)=\emptyset$;
    \item[2.] if $\scrI(b_n,0)=\emptyset$, then $( b_n \cdot \mu_{\scrI_0})$ converges to the unique probability Haar measure supported on 
    \begin{equation*}
        (\oplus_{I\in \scrI_1} \rmM_{I}) \rmU_{\scrI_1}\Gamma/\Gamma.
    \end{equation*}
\end{itemize}
\end{thm}

\section{Translates of horocyclic measures}

\subsection{Proof of Theorem \ref{thmMain}, nondivergence}

Notations same as section \ref{secMain}.
By \cite[Theorem 1.3]{ZhangArxivBorelSerre} (see also section 4 therein), 
\begin{thm}\label{thmNondivCri}
Assume $(\oplus a_n^k b_n)$ is clean.
The sequence of probability measures $ (\oplus a_n^k b_n \cdot \nu_{\scrI_0})$ is non-divergent iff there exists $\ep>0$ such that
for every $\Q$-subsapce $W$ stabilized by $\bmK_{\scrI_0}\bmU_{\scrI_0}$, 
\begin{equation*}
    \norm{\oplus a_n^k b_n \cdot (W\cap \Z^N) } \geq \ep
\end{equation*}
where $\norm{\cdot }$ denotes the covolume of a discrete subgroup in the $\R$-linear subspace spanned by it.
\end{thm}

So we need to classify all $\bmK_{\scrI_0}\bmU_{\scrI_0}$-stable subspaces. 

\begin{lem}
A proper subspace $W$ of $\R^N$ is $\bmU_{\scrI_0}$-stable iff there exists $j$ in $\{1,...,k_0-1\}$ such that 
\begin{equation*}
    \R^{I_1\cup...\cup I_{j}} \subset W \subsetneqq  \R^{I_1\cup...\cup I_{j+1}}.
\end{equation*}
\end{lem}

\begin{proof}
That if $W$ satisfies the sandwich condition then $W$ being $\bmU_{\scrI_0}$-stable  follows directly from definition.
Assume $W$ is $\bmU_{\scrI_0}$-stable, find the largest $k$ such that there exists $w \in W\setminus \R^{I_1\cup...\cup I_{k}}$. Then $w\in \R^{I_1\cup...\cup I_{k+1} }
\setminus  \R^{I_1\cup...\cup I_{k}}$ and from the definition of $\rmU_{\scrI_0}$, the map $u \mapsto uw -w$ is a surjection from $\rmU_{\scrI_0}$ to $\R^{I_1\cup...\cup I_k}$. Hence 
\begin{equation*}
    \R^{I_1\cup...\cup I_{k}} \subset W \subset \R^{I_1\cup...\cup I_{k+1}}.
\end{equation*}
Taking $j:=k$ or $k+1$ completes the proof.
\end{proof}

Since for each $k$, $\rmK_{I_k}$ acts on $\R^{I_k}$ irreducibly, we conclude that
\begin{lem}
A proper subspace $W$ of $\R^N$ is $\bmK_{\scrI_0}\bmU_{\scrI_0}$-stable iff there exists $j \in \{1,...,k_0-1\}$ such that 
\begin{equation*}
    W =\R^{I_1\cup...\cup I_{j}}.
\end{equation*}
\end{lem}

Now we can finish the proof of Theorem \ref{thmMain}, part 1.
For each $n$ and $j\in \{1,...,k_0-1\}$,
\begin{equation*}
        \norm{\oplus a_n^k b_n \cdot \Z^{I_1\cup...\cup I_{j}} }
        = \norm{b_n \cdot \Z^{I_1\cup...\cup I_{j}} }
        = \lambda_{I_1\cup...\cup I_{j}}(b_n).
\end{equation*}
So we are done by combining Theorem \ref{thmNondivCri} and the lemma just above.

\subsection{Proof of Theorem \ref{thmMain}, equidistribution}

Same notations as in Theorem \ref{thmMain}. Assume part 1 holds. 
Note that for every $n$,
\begin{equation*}
    \oplus a_n^k b_n (\rmK_{\scrI_0}\rmU_{\scrI_0})
    \subset (\oplus_{I\in \scrI_1(0)} \rmK_{I}) 
        (\oplus_{I\in \scrI_1(\infty)} \rmM_{I}) \rmU_{\scrI_1}.
\end{equation*}
Hence it suffices to work inside the latter, denoted by $\rmF_{\scrI_1}$ ($\bmF_{\scrI_1}$ the corresponding connected $\Q$-subgroup) for simplicity. 
By nondivergence, there exist sequences  $(k_n)$ in $\rmK_{\scrI_0}$, $(u_n)$ in $\rmU_{\scrI_0}$, bounded $(\delta_n)$ in $\rmF_{\scrI_1}$ and $(\gamma_n)$ in $\rmF_{\scrI_1}\cap \Gamma$, such that 
\begin{equation*}
    \oplus a_n^k b_n k_n u_n = \delta_n \gamma_n.
\end{equation*}
By passing to a finite cover, we assume 
\begin{equation*}
    \gamma_n = \oplus_{I\in \scrI_1(\infty)} \gamma^I_n \cdot \gamma_{\rmU,n}
\end{equation*}
for some $\gamma_n^I \in \rmM_{I}\cap \Gamma$ and $\gamma_{\rmU,n} \in \rmU_{\scrI_1}\cap \Gamma$.

By \cite[Theorem 2.1]{EskMozSha96} and passing to a subsequence if necessary, to prove Theorem \ref{thmMain}, it suffices to show that the only connected $\Q$-subgroup of $\bmF_{\scrI_1}$ containing 
\begin{equation}\label{equaEMS}
    \gamma_n \bmK_{\scrI_0} \bmU_{\scrI_0} \gamma_n^{-1}
\end{equation}
for all $n$ is $\bmF_{\scrI_1}$.
Since $\bmU_{\scrI_1}$ is contained in $\bmU_{\scrI_0}$ and $\bmU_{\scrI_1}$ is normalized by $\bmF_{\scrI_1}$, we have 
\begin{equation*}
      \bmU_{\scrI_1}\subset \gamma_n \bmK_{\scrI_0} \bmU_{\scrI_0} \gamma_n^{-1}
\end{equation*}
for all $n$. On the other hand, because of the special form $\gamma_n$ takes,
\begin{equation*}
      \oplus_{I\in \scrI_1(0)}\bmK_{I}
      \subset \gamma_n \bmK_{\scrI_0} \bmU_{\scrI_0} \gamma_n^{-1}
\end{equation*}
for all $n$.
Using the special form of $\gamma_n$ again, to prove our theorem, it suffices to show that for every $J \in \scrI_1(\infty)$, the union of 
\begin{equation*}
    \gamma^J_n (\oplus_{I \in \scrI_0, I \subset J  } \bmK_{I})
     (\bmU_{\scrI_0}\cap \bmM_J) (\gamma^J_n)^{-1}
\end{equation*}
is $\Q$-Zariski dense in $\bmM_J$.
By definition of $\scrI_1$, $a_n \in \rmM_{\scrI_1}$. Hence we can write $a_n = \oplus_{J\in\scrI_1} a_n^J$.
By repeating some arguments above if possible, we may assume for each $J\in \scrI_1(\infty) $,
\begin{equation*}
    (\oplus_{I \in \scrI_0, I \subset J  } a_n^I) b_n^J
    k_n^J u_n^J = \delta_n^J \gamma_n^J
\end{equation*}
for some sequences $(k_n^J)$ in $\oplus_{I \in \scrI_0, I \subset J  } \bmK_{I}$($=:\bmK_{\scrI_0,J}$), $(u_n^J)$ in $ \bmU_{\scrI_0}\cap \bmM_J$ and bounded $(\delta_n^J)$ in $\rmM_J$. 

There are two cases to consider: $J \in \scrI_1(\old,\infty)$ or $J \in \scrI_1(\new)$. The first case is simpler. Indeed, in this case $J \in \scrI_0$ and hence $(b_n^J)=(id)$, $ \bmU_{\scrI_0}\cap \bmM_J=\{id\}$. So we are considering the limit of (note that the condition in Equa.(\ref{equaEMS}) is not only sufficient, but also necessary)
\begin{equation*}
    a_n^J \widehat{\rmm}_{\rmK_{\scrI_0,J}} \text{ in } \rmM_J/\rmM_J\cap \Gamma
\end{equation*}
where $\widehat{\rmm}_{\rmK_{\scrI_0,J}} $ is the unique $\rmK_{\scrI_0,J}$-invariant probability measure supported the closed orbit of $\rmK_{\scrI_0,J}$ passing through the identity coset. Hence it converges to the $\rmM_J$-invariant Haar measure as $(a_n^J)$ is unbounded (see \cite{EskMcM93}). 

Now let us assume  $J \in \scrI_1(\new)$. Identify $\rmM_J = \SL_{|J|}(\R)$ and $\scrI_0':=\scrI_0\vert_{J}=\{I'_1,...,I'_{k_0'}\}$, a partition of $\{1,...,|J|\}$. By assumption, $k_0'\geq 2$.
Under this identification,
$\rmM_J \cap \Gamma$ is some lattice $\Gamma_J$ in $\SL_{|J|}(\R)$ commensurable with $\SL_{|J|}(\Z)$, $\rmU_{\scrI_0}\cap \rmM_J= \rmU_{\scrI_0'}$ and $\rmK_{\scrI_0,J}$ becomes $\rmK_{\scrI_0'}$. 

The sequence $ (\oplus_{I \in \scrI_0, I \subset J  } a_n^I )$ becomes
$( \oplus_{I \in  \scrI_0' } a_n^I)$, contained in 
$\oplus_{I \in  \scrI_0'} \rmA^+_{I}$. And $b_n^J$ is in $\rmA_{\scrI'_0}$. By the definition of $\scrI_1(\new)$ we have
\begin{equation*}
    \lambda_{I'_1\cup ...\cup I'_{k}}(b_n^J) \to +\infty,\quad
    \forall \, k=1,...,k_0'-1.
\end{equation*}
In particular, for all $l=1,...,|J|-1$,
\begin{equation*}
    \lambda_{\{1,...,l\}}(b_n^J) \to +\infty.
\end{equation*}
Also recall
\begin{equation*}
    \oplus_{I \in  \scrI_0' } a_n^I b_n^Jk_n^J u_n^J =\delta_n^J \gamma_n^J.
\end{equation*}
We would like to show that the only $\Q$-subgroup containing 
\begin{equation*}
    \gamma_n^J \bmK_{\scrI_0'} \bmU_{\scrI_0'} 
    (\gamma_n^J)^{-1}
\end{equation*}
is $\SL_{|J|}$. To do this, it suffices to show that for every nontrivial irreducible $\Q$-representation $(\rho, V)$ of $\SL_{|J|}$, rational vectors  fixed by $\gamma_n^J \bmK_{\scrI_0'} \bmU_{\scrI_0'} (\gamma_n^J)^{-1}$ for all $n$ must be zero. Now assume otherwise and we seek to derive a contradiction.

Let $\rmU_{\min}$ be the group of all upper triangular unipotent matrices in $\SL_{|J|}(\R)$. Let $\bmD$ be the full diagonal torus in $\SL_{|J|}$. Let $\psi^+\in X^*(\bmD)$ (the character group of $\bmD$) be the highest weight appearing in $(\rho, V)$. The ``highest'' is with respect to the partial order compatible with $\rmU_{\min}$. 

Let $\calW_{\scrI_0'}$ be the subgroup of the Weyl group generated by the reflection about the roots
\begin{equation*}
    \{\alpha_{i,j} \;\vert\; i\sim_{\scrI_0'}j \}
\end{equation*}
where $\alpha_{i,j}$ is defined by
\begin{equation*}
    \alpha_{i,j}(\diag(d_1,...,d_{|J|})):= d_i/d_j.
\end{equation*}
Note that $\calW_{\scrI_0'}$ admits a set of representatives
\begin{equation*}
    \{w_1,...,w_{s_1}\} \subset \rmK_{\scrI_0'}.
\end{equation*}

Let $\calC_{\scrI_0'}$ be the weights appearing in $(\rho,V)$ that lie in the convex hull of 
\begin{equation*}
    \{ w\cdot \psi^+ \;\vert\; w\in \calW_{\scrI_0'} \}.
\end{equation*}
Then $\oplus_{\alpha\in \calC_{\scrI_0'}} V_{\alpha}$ is fixed by $\bmU_{\scrI_0'}$. Indeed 
\begin{lem}
We have $\oplus_{\alpha\in \calC_{\scrI_0'}} V_{\alpha}=V^{\bmU_{\scrI_0'}}$. 
\end{lem}

\begin{proof}
This has essentially been done in \cite{Shi19}. Let us briefly recall how. 
Both $\oplus_{\alpha\in \calC_{\scrI_0'}} V_{\alpha}$ and $V^{\bmU_{\scrI_0'}}$ are $\bmM_I$-stable. Thus we can find a $\bmM_I$-stable complement $W$ of $\oplus_{\alpha\in \calC_{\scrI_0'}} V_{\alpha}$ in $V^{\bmU_{\scrI_0'}}$. Hence $W$ contains a vector fixed by $\bmU_{\min}$, which is a contradiction.
\end{proof}

We do not plan to characterize $V^{\bmK_{\scrI_0'}\bmU_{\scrI_0'}}$. Rather we show the following
\begin{lem}
There exists $C_2>0$ such that for all 
$v \in  V^{\bmK_{\scrI_0'}\bmU_{\scrI_0'}}$ and $d \in \rmA_{\rmM_{\scrI_0'}}^+$, we have
\begin{equation*}
    \norm{\rho(d)v}\geq C_2\norm{v}.
\end{equation*}
\end{lem}

\begin{proof}
Without loss of generality assume different weight spaces are orthogonal to each other.

Decompose $\calC_{\scrI_0'}=\calC_1\sqcup...\sqcup\calC_r$ into disjoint $\calW_{\scrI_0'}$-orbits.
For each $i=1,..,r$, let $\psi_i^+$ be the unique element in $\calC_i$ that are in the positive Weyl chamber.

Take a vector $v\in V^{\bmK_{\scrI_0'}\bmU_{\scrI_0'}}$. Write $v=v^1\oplus v^2 \oplus ...\oplus v^{r}$ such that each 
\begin{equation*}
    v^i = \oplus_{\psi \in \calC_i} v^i_{\psi} \in \oplus_{\psi \in \calC_i} V_{\psi}.
\end{equation*}

Because $v$ is fixed by $\bmK_{\scrI_0'}$, 
$\{w_1,...,w_{s_1}\}$ transitively permutes $\{v^i_{\psi}\}_{\psi \in \calC_i}$.
Thus there exists a constant $C_1>0$ such that 
\begin{equation*}
    \norm{v^i_{\psi}} \geq C_1\norm{v^i},\;\forall i=1,...,r,\;\forall \psi \in\calC_i.
\end{equation*}
Foe every $i$, $\psi_i^+$ is a positive linear combinations of $\lambda_{\{1\}},...,\lambda_{\{1,...,|J|-1\}}$. As there are only finitely many such, we can find  $C_1'>0$ such that for all $d\in \rmA_{\rmM_{\scrI_0'}}^+$,
\begin{equation*}
    \left|\psi^+_i(d)\right| \geq C_1'.
\end{equation*}

Now 
\begin{equation*}
    \begin{aligned}
     \norm{\rho(d)v}^2 =& \sum_{i=1}^r \norm{\rho(d)\oplus_{\psi \in\calC_i} v^i_{\psi}}^2 \\
     \geq& \sum_{i=1}^r \norm{\rho(d) v^i_{\psi_i^+}}^{2}
     = \sum_{i=1}^r |\psi_i^+(d)|^{2} \norm{v^i_{\psi_i^+}}^2
     \\
     \geq& C_1'^2  \sum_{i=1}^r C_1^2 \norm{v^i}^2 =  (C_1C_1')^2 \norm{v}^2.
    \end{aligned}
\end{equation*}
Taking $C_2:=C_1C_1'$ completes the proof.
\end{proof}

Now start the proof. Let $v$ be a rational vector fixed by
$ \gamma_n^J \bmK_{\scrI_0'} \bmU_{\scrI_0'} (\gamma_n^J)^{-1}$. 
Because $v_n:=\rho(\gamma_n^J)^{-1}v$ have bounded denominators, there exists $C_3>0$ such that 
    \begin{equation*}
        \norm{v_n} \geq C_3,\;\forall\,n.
    \end{equation*}
Also $v_n$ and hence $\rho(b_n^J)v_n$ are fixed by  $\bmK_{\scrI_0'} \bmU_{\scrI_0'}$. Thus (for simplicity $\rho$ is dropped from the notations below)
\begin{equation*}
    \begin{aligned}
       \norm{\delta_n^J\cdot v}= \norm{ \delta_n^J \gamma_n^J \cdot  v_n}
       = \norm{ \oplus a_n^I b_n^J  (k_n^Ju_n^J)\cdot 
       v_n}
       = \norm{\oplus a_n^I  b_n^J  \cdot 
       v_n}
       \geq C_2 \norm{  b_n^J  \cdot 
       v_n}.
    \end{aligned}
\end{equation*}
For every $\psi \in \calC_{\scrI_0'}$, its restriction to $\rmA_{\scrI_0'}$ is the same as the restriction of $\psi^+$. But $\psi^+(b_n^J ) \to +\infty$ by assumption.
Hence
\begin{equation*}
    \norm{ b_n^J  \cdot 
       v_n} =\psi^+(b_n^J ) \norm{v_n} \geq \psi^+(b_n^J ) C_3 \to +\infty.
\end{equation*}
This is a contradiction since $\norm{\delta_n^J\cdot v}$ is bounded.

\section{Volume computation and count}\label{secCount}

In this section we apply equidistribution of homogeneous measures to count lifts of certain horocycles.
Notations are inherited from section \ref{secMain}.
Write $\rmG:=\SL_N(\R)$ and $\rmK:=\SO_N(\R)$. Gothic letters are used to denote Lie algebras.
Let $\Gamma\leq \SL_N(\R)$ be a lattice commensurable with $\SL_N(\Z)$. 

To simplify notations, in a space $X$ equipped with the action of a group $H$, we use $[x]_H$ to denote the image of $x\in X$ in $X/H$.
E.g., for $g\in \rmG$, $[g]_{\Gamma}$ is its image in $\rmG/\Gamma$.

\subsection{Statement of the theorem with proof outlined}

Proofs of lemmas in this subsection are often delayed to a later subsection. Links are provided to facilitate readers of the digital version to navigate back and forth.

Let $\Gamma \bs \Gamma\rmU_{\scrI_0}\rmK/\rmK$  be a closed horocycle in the locally symmetric space $\Gamma \bs \rmG/\rmK $ for some partition $\scrI_0$ of $\{1,...,N\}$.
Let $G_{\hor}$ be the stabilizer in $\rmG$ of the (lifted) horocycle $\rmU_{\scrI_0}\rmK/\rmK$ in $\rmG/\rmK$. Thus the set of all lifts of $\Gamma \bs \Gamma\rmU_{\scrI_0}\rmK/\rmK$ in $\rmG/\rmK$ is parametrized by 
$\Gamma/  G_{\hor}\cap\Gamma $ via 
\begin{equation*}
    \gamma \mapsto \gamma \rmU_{\scrI_0}\rmK/\rmK.
\end{equation*}

\begin{lem}\label{lemGhoroDescription}
We have that $G_{\hor}= N_{\rmK}(\rmU_{\scrI_0}) \cdot \rmU_{\scrI_0}$ and hence
$G_{\hor}^{\circ} = \rmK_{\scrI_0} \cdot \rmU_{\scrI_0}$. Here 
$ N_{\rmK}(\rmU_{\scrI_0})$ denotes the normalizer of 
$ \rmU_{\scrI_0}$ in $\rmK$.
\end{lem}

See \ref{prooflemGhoro} for the proof.

 Now choose $\Gamma'$ to be a \textit{neat} finite index subgroup of $\Gamma$. Being neat means that for every $\gamma \in \Gamma'$, 
 no eigenvalue of $\gamma$ is a root of unity except for $ 1$. 
 Such a choice of $\Gamma'$ always exists (\cite[17.4]{Bor2019}).
 There exists $\{q_1,...,q_{r_0}\}\subset \Gamma$  such that 
 \begin{equation*}
 \begin{aligned}
     \Gamma \rmU_{\scrI_0}\rmK/\rmK =
     \bigsqcup_{i=1}^{r_0} \Gamma'q_i \rmU_{\scrI_0}\rmK/\rmK 
     = \bigsqcup_{i=1}^{r_0} q_i \Gamma_i \rmU_{\scrI_0}\rmK/\rmK 
 \end{aligned}
 \end{equation*}
 where $\Gamma_i:= q_i^{-1}\Gamma'q_i$. Note that each $\Gamma_i$ is also neat.
 
 \begin{lem}\label{lemGammaGhoro}
 For each $i=1,...,r_0$, $ G_{\hor}\cap \Gamma = \rmU_{\scrI_0}\cap \Gamma$.
 \end{lem}
  
 See \ref{proofGammaGhoro} for the proof.
 
 Thus the set of lifts of $\Gamma \bs \Gamma\rmU_{\scrI_0}\rmK/\rmK$  is parametrized by 
 \begin{equation*}
     \bigsqcup_{i=1,...,r_0} q_i\Gamma_i/\rmU_{\scrI_0}\cap\Gamma_i. 
 \end{equation*}
 
 To count these lifts, we consider the function $\height: \rmG/G_{\hor}\to \R$ defined as
 \begin{equation*}
     \height([g]_{G_{\hor}}):=
     \dist(g \rmU_{\scrI_0}\rmK/\rmK , [id]_{\rmK}).
 \end{equation*}
 With respect to this function, define $B_R\subset \rmG/G_{\hor}$ as 
 \begin{equation*}
     B_R :=\{ [g]_{G_{\hor}}\;\vert\; \height([g])\leq R\}.
 \end{equation*}
 We will actually use $\wtB_R$, the preimage of $B_R$ in $\rmG/G_{\hor}^{\circ}$, later. 
 
 For each $i=1,...,r_0$ and $R>0$, consider the function
 $  \varphi_R^{i}: \rmG/\Gamma_i \to \R$ defined by
 \begin{equation*}
 \begin{aligned}
     \varphi_R^{i}([g]_{\Gamma}):=&
     \frac{1}{C^i_R} \sum_{\gamma \in \Gamma_i /\rmU_{\scrI_0}\cap\Gamma_i} 1_{B_R}([g\gamma]_{G_{\hor}})\\
     =& \frac{1}{C^i_R} \sum_{\gamma \in \Gamma_i /\rmU_{\scrI_0}\cap\Gamma_i} 1_{\wtB_R}([g\gamma]_{G^{\circ}_{\hor}})
 \end{aligned}
 \end{equation*}
 for certain sequence $(C^i_R)$ of positive real numbers to be defined soon (see Equa.(\ref{equaDefineC_Ri})).
 
 From the definition, if as $R$ tends to infinity, $\varphi^i_R([q_i])$ converges to $1$ for each $i$ then we could conclude that 
 \begin{equation*}
     \#
     \left\{ \gamma \rmU_{\scrI_0}\rmK/\rmK \;\middle\vert\;
     \dist( \gamma \rmU_{\scrI_0}\rmK/\rmK , [id]_{\rmK})\leq R
     \right\} \sim 
     \sum_{i=1,...,r_0} C_R^i.
 \end{equation*}
 
 For simplicity write 
 \begin{equation*}
     \mu_{\rmA}:=\prod_{i<j,i\sim_{\scrI_0} j}
     \frac{\alpha_{ij}(a)-\alpha_{ji}(a)}{2}
       \alpha_{\scrI_0}(b)\cdot
      \rmm_{ \rmA^+_{\rmM_{\scrI_0}}}(a)
      \otimes \rmm_{\rmA_{\scrI_0}}(b),
 \end{equation*}
 a measure on $\rmA_{\rmM_{\scrI_0}}^+ \oplus \rmA_{\scrI_0}$.
 Recall for $a=\diag(a_1,...,a_N)$, $\alpha_{ij}(a):=a_i/a_j$ and $\alpha_{\scrI_0}(a):= \prod_{i<j,i\not\sim_{\scrI_0} j} \alpha_{ij}(a)$.
 By abuse of notation we think of $\mu_A$ also as a measure on
 $\fraka_{\rmM_{\scrI_0}}^+ \oplus \fraka_{\scrI_0}$ where $\fraka_{\rmM_{\scrI_0}}^+:=\log(\rmA_{\rmM_{\scrI_0}}^+)$.
 
 \begin{lem}\label{lemDecompHaar}
  For 
  \begin{equation*}
      C_7:=\Vol(\rmK)\cdot 2^{-\frac{\sum_{i<j}|I_i||I_j|}{2}}
  \end{equation*}
  the surjective map
  \begin{equation*}
      \Phi_7: \rmK \times \rmA_{\rmM_{\scrI_0}}^+ \times \rmA_{\scrI_0} \to \rmG/\rmG^{\circ}_{\hor}
  \end{equation*}
  defined by group multiplication followed by a natural projection induces 
  \begin{equation*}
      (\Phi_7)_*
      \left(
      C_7 \cdot
      \widehat{\rmm}_{\rmK} \otimes \mu_{\rmA}
      \right) =\rmm_{\rmG/\rmG_{\hor}^{\circ}}.
  \end{equation*}
 \end{lem}
 See \ref{proofDecompHaar} for the proof and the definition of $\rmm_{\rmG/G^{\circ}_{\hor}}$.
  For $R>0$, write 
  \begin{equation*}
  \begin{aligned}
            &B_{\fraka}(R):= \left\{
       (x,y)\in \fraka_{\rmM_{\scrI_0}}^+ \oplus \fraka_{\scrI_0}\;\middle\vert\;
       \norm{x+y}\leq R
       \right\}, \\
       &B^+_{\fraka}(R):= \left\{
       (x,y)\in \fraka_{\rmM_{\scrI_0}}^+ \oplus \fraka_{\scrI_0}^+\;\middle\vert\;
       \norm{x+y}\leq R
       \right\}.
  \end{aligned}
  \end{equation*}
  Here 
\begin{equation*}
  \begin{aligned}
       \fraka_{\scrI_0}^+:=\left\{ \bma\in \fraka_{\scrI_0}
       \;\middle\vert\;
       \sum_{i=1}^k a_i |I_i| \geq 0, \;\forall k=1,...,k_0-1;\;
       \sum_{i=1}^{k_0} a_i |I_i| = 0
       \right\}
  \end{aligned}
  \end{equation*}
  where we have written
  \begin{equation*}
      \bma= \diag(a_1 id_{|I_1|},...,a_{k_0}id_{|I_{k_0}|})\in \fraka_{\scrI_0}.
  \end{equation*}
  
  \begin{lem}\label{lemHtBallDecomp}
   The height ball pulls back to
   \begin{equation*}
       (\Phi_7)^{-1} (\wtB_R)
       =\rmK \times \exp(B_{\fraka}(R)).
   \end{equation*}
  \end{lem}
  
   See \ref{proofHtBallDecomp} for the proof.
   Now we define 
   \begin{equation}\label{equaDefineC_Ri}
       C_R^i:=
         C_7\mu_{\rmA}\left(B^+_{\fraka}(R)\right)
        \frac{\Vol(G^{\circ}_{\hor}/\rmU_{\scrI_0}\cap\Gamma_i)}{\Vol(\rmG/\Gamma_i)}.
   \end{equation}
   Let $\bmv_0^+$ be the sum of positive roots. Via the trace form, it is identified with $\bmv_0$ in the proof of the lemma below (see \ref{proofVolumeComp}). 
   
   \begin{lem}\label{lemVolumeComp}
   As $R$ tends to $+\infty$, we have
   \begin{equation*}
       C^i_R \sim 
       \Vol(\rmK)\left(\frac{1}{2}\right)^{
       \frac{N(N-1)}{2}
       }\left(
       \frac{2\pi R}{\norm{\bmv_0^+}}
       \right)^{\frac{N-2}{2}} e^{\norm{\bmv_0^+}R}\cdot
       \frac{\Vol(G^{\circ}_{\hor}/\rmU_{\scrI_0}\cap\Gamma_i)}{\Vol(\rmG/\Gamma_i)}.
   \end{equation*}
   \end{lem}
   
   See \ref{proofVolumeComp} for the proof. 
   Now we can state the counting theorem
   \begin{thm}\label{thmCount}
   The asymptotic count of lifts of a closed horocycle is given by
     \begin{equation*}
         \begin{aligned}
            & \#
     \left\{ \gamma \rmU_{\scrI_0}\rmK/\rmK \;\middle\vert\;
     \dist( \gamma \rmU_{\scrI_0}\rmK/\rmK , [id]_{\rmK})\leq R,\;\gamma \in \Gamma
     \right\} \\
     &\sim 
     \left(\frac{1}{2}\right)^{\frac{N(N-1)}{2}} 
     \left( \frac{2\pi R}{\norm{\bmv_0^+}}
     \right)^{\frac{N-2}{2}} e^{\norm{\bmv_0^+}R}
     \cdot \frac{\Vol(G_{\hor}/G_{\hor}\cap \Gamma)/(2^{k_0}-1)}{\Vol(\rmK\bs\rmG/\Gamma)}.
         \end{aligned}
     \end{equation*}
     And if $\Gamma$ is neat, it can be rewritten as 
     \begin{equation*}
         \begin{aligned}
            & \#
     \left\{ \gamma \rmU_{\scrI_0}\rmK/\rmK \;\middle\vert\;
     \dist( \gamma \rmU_{\scrI_0}\rmK/\rmK , [id]_{\rmK})\leq R,\;\gamma \in \Gamma
     \right\} \\
     &\sim 
     \left(\frac{1}{2}\right)^{\frac{N(N-1)}{2}} 
     \left( \frac{2\pi R}{\norm{\bmv_0^+}}
     \right)^{\frac{N-2}{2}} e^{\norm{\bmv_0^+}R}
     \cdot \frac{\Vol(\rmU_{\scrI_0}/\rmU_{\scrI_0}\cap \Gamma)\Vol(\rmK_{\scrI_0})}{\Vol(\rmK\bs\rmG/\Gamma)}.
         \end{aligned}
     \end{equation*}
   \end{thm}

   \begin{proof}[Proof assuming Lemma \ref{lemWellRound} and Proposition \ref{propMeanCount} ]
   Indeed it follows from  Lemma \ref{lemWellRound} and Proposition \ref{propMeanCount} in the next section that
    \begin{equation*}
     \#
     \left\{ \gamma \rmU_{\scrI_0}\rmK/\rmK \;\middle\vert\;
     \dist( \gamma \rmU_{\scrI_0}\rmK/\rmK , [id]_{\rmK})\leq R
     ,\;\gamma\in\Gamma 
     \right\} \sim 
     \sum_{i=1,...,r_0} C_R^i .
 \end{equation*}
   It remains to sum $C_R^i$'s together. Call $C_{R,0}$ the terms in Lemma \ref{lemVolumeComp} that are independent of $i$, i.e.
   \begin{equation*}
       C_R^i = C_{R,0}\cdot \frac{\Vol(G^{\circ}_{\hor}/\rmU_{\scrI_0}\cap\Gamma_i)}{\Vol(\rmG/\Gamma_i)}.
   \end{equation*}
   Recall $\Gamma_i=q_i^{-1}\Gamma' q_i$ for some $q_i\in\Gamma$ such that
   \begin{equation*}
       \Gamma=\bigsqcup_{i=1}^{r_0} \Gamma'q_i (G_{\hor}\cap\Gamma).
   \end{equation*}
   Regarding this as a decomposition of $\Gamma'\bs \Gamma$ into disjoint $G_{\hor}\cap\Gamma$ -orbits, we have for every $j=1,...,r_0$,
   \begin{equation*}
      [\Gamma:\Gamma_j]= [\Gamma:\Gamma']=\sum_{i=1}^{r_0} 
       [G_{\hor}\cap \Gamma:G_{\hor}\cap \Gamma_i ]
       =\sum_{i=1}^{r_0} [G_{\hor}\cap \Gamma:\rmU_{\scrI_0}\cap \Gamma_i ].
   \end{equation*}
   Let $\pi_0$ be the set of connected components, then by Lemma \ref{lemGammaGhoro}
   \begin{equation*}
       \begin{aligned}
           \sum_{i=1}^{r_0} C^i_R =&
            \sum_{i=1}^{r_0} C_{R,0}\cdot \frac{\Vol(G_{\hor}/\rmU_{\scrI_0}\cap\Gamma_i)}{|\pi_0(G_{\hor})|\Vol(\rmG/\Gamma_i)}\\
            =& \sum_{i=1}^{r_0} C_{R,0}\cdot \frac{
            \Vol(G_{\hor}/G_{\hor}\cap \Gamma)
            [G_{\hor}\cap \Gamma:\rmU_{\scrI_0}\cap\Gamma_i]
            }
            {|\pi_0(G_{\hor})|\Vol(\rmG/\Gamma)[\Gamma:\Gamma_i]}\\
            =&C_{R,0}\cdot \frac{
            \Vol(G_{\hor}/G_{\hor}\cap \Gamma)
            }
            {|\pi_0(G_{\hor})|\Vol(\rmG/\Gamma)}.
       \end{aligned}
   \end{equation*}
   In our case, $\pi_0(G_{\hor})=\pi_0(N_{\rmG}\rmU_{\scrI_0}\cap \rmK)$ and the latter of which can be identified as
   \begin{equation*}
   \begin{aligned}
        &S(O_{|I_1|}(\R)\times ...\times O_{|I_{k_0}|}(\R))\\
        &:=
       \left\{
       (M_1,...,M_{k_0}) \in O_{|I_1|}(\R)\times ...\times O_{|I_{k_0}|}(\R)\;\middle\vert\;
       \prod_{i=1}^{k_0}\det(M_i)=1
       \right\}.
   \end{aligned}
   \end{equation*}
   Hence $|\pi_0(G_{\hor})|=2^{k_0}-1$.
   \end{proof}
   
  Note that (see \ref{proofVolumeComp} for the explicit expression of $\bmv_0$)
   \begin{equation*}
       \norm{\bmv^+_0}=\norm{\bmv_0}=\sqrt{\sum_{i=1}^N (N-2i+1)^2}
   \end{equation*}
   admits an explicit expression. Call this $P_N$ for simplicity.
   \begin{lem}\label{lemNormV0}
   For $N\in \Z_{>0}$, we have the identity
   \begin{equation*}
       P_N^2={\sum_{i=1}^N (N-2i+1)^2} = \frac{1}{3}N(N-1)(N+1).
   \end{equation*}
   \end{lem}
   \begin{proof}
       By induction, for $k\in\Z_{>0}$,
   \begin{equation*}
       \sum_{l=1}^k l^2  = \frac{1}{6}k(k+1)(2k+1).
   \end{equation*}
   When $N=2n$ is even,
   \begin{equation*}
   \begin{aligned}
       {}&\sum_{i=1}^N (N-2i+1)^2 = 2 \sum_{i=1}^{n}(2i-1)^2
       =2 \left(
       -\sum_{i=1}^n (2i)^2 + \sum_{i=1}^{2n} i^2
       \right) \\
       &=\frac{2}{6}
       \left(
       -4n(n+1)(2n+1) + 2n(2n+1)(4n+1)
       \right)\\
       &=\frac{1}{3}
        2n \left( 
       -(2n+2)(2n+1) + (2n+1)(4n+1)
       \right)
       =\frac{1}{3} 2n (2n+1) (2n-1).
   \end{aligned}
   \end{equation*}
   When $N=2n+1$ is odd,
   \begin{equation*}
       \begin{aligned}
           \sum_{i=1}^N (N-2i+1)^2= 
           2 \sum_{i=1}^{n} (2i)^2 = \frac{8}{6}n(n+1)(2n+1)
           =\frac{1}{3}2n(2n+2)(2n+1).
       \end{aligned}
   \end{equation*}
   So the proof completes.
   \end{proof}
   
    In the case $\Gamma=\SL_N(\Z)$, the rest of the expressions can also be worked out explicitly. Let $n_k:=|I_k|$.
   
   \begin{coro}\label{coroCountSLZ}
      We have
      \begin{equation*}
      \begin{aligned}
          &\#
     \left\{ \gamma \rmU_{\scrI_0}\rmK/\rmK \;\middle\vert\;
     \dist( \gamma \rmU_{\scrI_0}\rmK/\rmK , [id]_{\rmK})\leq R,\;
      \gamma \in \SL_N(\Z)
     \right\} \\
     \sim &
     \left(\frac{1}{2}\right)^{\frac{N(N-1)}{2}} 
     \left( \frac{2\pi R}{P_N}
     \right)^{\frac{N-2}{2}} e^{P_NR}
     \frac{1}{2^{k_0}-1}
     \prod_{k=1}^{k_0} \frac{\Vol(\SO_{n_k}(\R))
     }
     {n_k! 2^{n_k-1}
     }
     \frac{\Vol(\SO_N(\R))}{\Vol(\SL_N(\R)/\SL_N(\Z))}.
      \end{aligned}
      \end{equation*}
      \end{coro}
      
      Here $P_N$ is as in Lemma \ref{lemNormV0}, $\Vol(\SO_n(\R))$ is as in Lemma \ref{lemVolCptLieGp}, and 
      \begin{equation*}
          \Vol(\SL_N(\R)/\SL_N(\Z))=\zeta(2)\zeta(3)\cdot ...\cdot \zeta(N)
      \end{equation*}
      where $\zeta$ is the Riemann zeta function.

   If one employs the completed Riemann zeta function 
   \begin{equation*}
       \xi(s):=\frac{1}{2}s(1-s)\pi^{-\frac{s}{2}}\Gamma(\frac{s}{2})\zeta(s),
   \end{equation*}
   one can write 
  \begin{equation*}
       \frac{\Vol(\SO_N(\R))}{\Vol(\SL_N(\R)/\SL_N(\Z))}
       =\frac{2^{\frac{N(N-1)}{4}}
       N!(N-1)! 
       }
       {\xi(2)\cdot\xi(3)\cdot...\cdot \xi(N)}.
  \end{equation*}
  
   \begin{proof}
    Note that $\SL_N(\Z)$ intersect nontrivially with every connected component of $N_{\rmK}(\rmU_{\scrI_0})$, thus
    \begin{equation*}
        \Vol(G_{\hor}/G_{\hor}\cap \SL_N(\Z)) =
        \Vol(G^{\circ}_{\hor}/G^{\circ}_{\hor}\cap \SL_N(\Z)).
    \end{equation*}
    Also one can check that
    \begin{equation*}
        G^{\circ}_{\hor}\cap \SL_N(\Z) = 
       ( \rmK_{\scrI_0} \cap \SL_N(\Z)) \ltimes (\rmU_{\scrI_0} \cap \SL_N(\Z))
    \end{equation*}
    and
    \begin{equation*}
        \# \rmK_{\scrI_0} \cap \SL_N(\Z)
        = \prod_{k=1}^{k_0} \#\SO_{n_k}(\Z) = \prod_{k=1}^{k_0} n_k! 2^{n_k-1}.
    \end{equation*}
    Thus 
    \begin{equation*}
    \begin{aligned}
         \Vol(G^{\circ}_{\hor}/G^{\circ}_{\hor}\cap \SL_N(\Z))
        =& \frac{\Vol(\rmK_{\scrI_0})}{\prod_{k=1}^{k_0} n_k!\cdot  2^{n_k-1}} \cdot \Vol(\rmU_{\scrI_0}/\rmU_{\scrI_0} \cap \SL_N(\Z))
    \end{aligned}
    \end{equation*}
    As $\left\{E_{ij},\;i<j,\;i\not\sim_{\scrI_0} j \right\}$
    forms an ortho-normal basis of $\fraku_{\scrI_0}$,  $\rmU_{\scrI_0}/\rmU_{\scrI_0} \cap \SL_N(\Z)$ has volume one
    and
     \begin{equation*}
         \Vol(G^{\circ}_{\hor}/G^{\circ}_{\hor}\cap \SL_N(\Z))
         =\frac{\Vol(\rmK_{\scrI_0})}{\prod_{k=1}^{k_0} n_k! 2^{n_k-1}}
         =\frac{\prod_{k=1}^{k_0}\Vol(\SO_{n_k}(\R))}{\prod_{k=1}^{k_0} n_k! 2^{n_k-1}}.
     \end{equation*}
     Finally the identity
     \begin{equation*}
          \Vol(\SL_N(\R)/\SL_N(\Z))=\zeta(2)\zeta(3)\cdot ...\cdot \zeta(N)
      \end{equation*}
      is classical. One may see \cite{LanglandsVolFund}, noting that the standard \textit{Chevalley basis} (see \cite[25.2]{HumphreysLieAlge}) forms an ortho-normal basis under the trace form and that $\SL_N$ is a simply connected algebraic group.
   \end{proof}

   \begin{proof}[Proof of Example 1(\ref{Example1}) and 2(\ref{Example2})]
    For these two cases, $N=3$. $$P_3^2=\frac{1}{3}3(3-1)(3+1)=8.$$
   The asymptotic in Coro.\ref{coroCountSLZ} excluding the terms concerning $\scrI_0$ is 
   \begin{equation*}
   \begin{aligned}
        R^{\frac{1}{2}}e^{2\sqrt{2}R}\cdot 
       2^{-3(3-1)/2}\cdot \pi^{\frac{1}{2}}2^{-\frac{1}{4}}
       \cdot \frac{2^{\frac{3(3-1)}{4}}
       3!(3-1)! 
       }
       {\xi(2)\xi(3)}
       =R^{\frac{1}{2}}e^{2\sqrt{2}R}
       \frac{
       \pi^{\frac{1}{2}}
       \cdot 3 \cdot 2^{\frac{1}{4}}
       }{\xi(2)\xi(3)}.
   \end{aligned}
   \end{equation*}
    Now we calculate the term depending on $\scrI_0$.
    For example 1, $\scrI_0=\{\{1\},\{2\},\{3\} \}$. Thus $k_0=3$ and all $n_k$'s are $1$ and the term is just ${1}/{(2^3-1)}=1/7$.
    For example 2, $\scrI_0= \{\{1,2\},\{3\} \}  $. Thus $k_0=2$ and $n_1=2$, $n_2=1$. So the term is 
    \begin{equation*}
        \frac{1}{2^2-1} \cdot \frac{2\sqrt{2}\pi}{2^2} = \frac{\pi}{3\cdot 2^{\frac{1}{2}}
        }. 
    \end{equation*}
   \end{proof}

   \subsection{Completion of the proof}
   
   As in \cite{DukRudSar93}, the pointwise convergence is deduced from a weak convergence plus the following 
   \begin{lem}\label{lemWellRound}
   The family of sets $(B_R)_{R\geq 1}$ is \textit{well-rounded} in the sense of \cite[Proposition 1.3]{EskMcM93}.
   \end{lem}
   
   For the precise definition of well-roundedness and the proof, see \ref{proofWellRound}.
   
   The weak convergence is the following
   \begin{prop}\label{propMeanCount}
     For every compactly supported function $\psi$ on $\rmG/\Gamma_i$, 
     \begin{equation}\label{EquaPropWeakConv}
         \lim_{R\to +\infty}
         \int \varphi^i_R(x) \psi(x) \rmm_{\rmG/\Gamma_i} (x)
         = \int \psi(x) \rmm_{\rmG/\Gamma_i}(x).
     \end{equation}
   \end{prop}
   
   The rest of this section is devoted to the proof of this proposition.
  
   For $R,\ep>0$, define 
   \begin{equation*}
       B^+_{\fraka}(R,\ep):= B^+_{\fraka}(R) \setminus B^+_{\fraka}(\ep R).
   \end{equation*}

   \begin{lem}
       Let $\Lambda$ be a lattice commensurable with $\SL_N(\Z)$.
       For every $\delta>0$, every sequence $(R_n)$ tending to infinity and $(g_n)$ with $g_n\in \rmK \times \exp( B^+_{\fraka}(R_n,\delta) )$,
       \begin{equation*}
           \lim_{n\to \infty} (g_n)_* \widehat{\nu}_{[\scrI_0]}= \widehat{\rmm}_{\rmG/\Lambda}.
       \end{equation*}
       where $\widehat{\mu}$ denotes the unique probability measure proportional to a finite measure $\mu$.
   \end{lem}
   
   \begin{proof}
   Follows directly from Theorem \ref{thmMain}. 
   \end{proof}
   
   In a different vein, for $C\in \R$, define 
   \begin{equation*}
         \calC_{C}:=
         \left\{\bma,\;
         a_i -a_j \geq\max\{0,C\},\;\forall k,\, i<j \in I_k;\;
         \sum_{i\in I_1\cup ...\cup I_k} a_{i} \geq C, \;\forall 1\leq k<k_0
         \right\}
     \end{equation*}
     where $\bma=\diag(a_1,...,a_{N})$ for an element $\bma\in \fraka^+_{\rmM_{\scrI_0}} \oplus \fraka_{\scrI_0}$. Thus $\calC_0$ is nothing but $\fraka_{\rmM_{\scrI_0}}^+ \oplus \fraka_{\scrI_0}^+$.
   
   \begin{lem}\label{lemDiverg}
   For every $i$ and every compact set $B$ of $\rmG/\Gamma_i$, there exists $C=C^i_B<0$ such that if 
   \begin{equation*}
       (k,a,b) \notin \rmK \times \exp(\calC_{C})
   \end{equation*}
   then
   \begin{equation*}
       [kab G_{\hor}]_{\Gamma_i} \cap B =\emptyset.
   \end{equation*}
   \end{lem}
   
   See \ref{proofDiverg} for a sketch of proof.
   Define for $C\in\R$ and $R>0$,
   \begin{equation*}
       B^{C,+}_{\fraka}(R):=B_{\fraka}(R)\cap \calC_C.
   \end{equation*}
   Asymptotically the measures of 
   $ B^{C,+}_{\fraka}(R),\,B^{+}_{\fraka}(R),\, B^+_{\fraka}(R,\ep)$ are not so different from each other. 
   
   \begin{lem}\label{lemAsymVolCompare}
   For every $C<0$,
   \begin{equation*}
       \lim_{R\to+\infty} 
       \frac{\mu_{\rmA}\left(B^{C,+}_{\fraka}(R)\right)}
       {\mu_{\rmA}\left(B^{+}_{\fraka}(R)\right)} =1.
   \end{equation*}
   For every $\ep\in(0,1)$, there exists $\delta_{\ep}>0$ such that 
   \begin{equation*}
       \limsup_{R\to+\infty} \frac{\mu_{\rmA}\left(B^{+}_{\fraka}(R)\right)}
       {\mu_{\rmA}\left(B^{+}_{\fraka}(R,\delta_{\ep})\right)} \leq \min\{1 +\ep, \frac{1}{1-\ep}\}.
   \end{equation*}
   \end{lem}
   
   See \ref{proofAsymVolCompare} for the proof.
   Now fix $\psi$, a nonzero compactly supported continuous function on $\rmG/\Gamma_i$. To establish Equa.(\ref{EquaPropWeakConv}), it suffices to show that for every $\ep>0$, 
   \begin{equation}\label{EquaPropWeakConvEps}
   \begin{aligned}
       &\liminf_{R\to+\infty} \int \varphi^i_R(x)\psi(x) \rmm_{\rmG/\Gamma_i}(x) \geq (1-\ep) \int \psi(x)\rmm_{\rmG/\Gamma_i}(x)-\ep\\
        &\limsup_{R\to+\infty} \int \varphi^i_R(x)\psi(x) \rmm_{\rmG/\Gamma_i}(x) \leq (1+\ep) \int \psi(x)\rmm_{\rmG/\Gamma_i}(x) + \ep.
   \end{aligned}
   \end{equation}
   Let $\norm{\psi}_{\sup}:=\sup\{|\psi(x)|,\,x\in \rmG/\Gamma_i\}$.
   Choose 
   \begin{equation*}
       \delta:=\delta_{\ep'} \text{ with }
       \ep':=\frac{\ep}
       {
       \norm{\psi}_{\sup} \Vol(\rmG/\Gamma_i)
       }
   \end{equation*}
   according to Lemma \ref{lemAsymVolCompare}. Choose $C:=C^i_B$ with $B:=\supp(\psi)$ according to Lemma \ref{lemDiverg}. For simplicity write $\Gamma_U:=\rmU_{\scrI_0}\cap\Gamma_i$.
   Also define the map $p,q$ by natural projections:
   \begin{equation*}
       \begin{tikzcd}
      &   \rmG/\Gamma_{U} \arrow[dl, swap,"p"] \arrow[dr, "q"]&\\
     \rmG/\Gamma_i &  & \rmG/G_{\hor}^{\circ}.
      \end{tikzcd}
   \end{equation*}
   
   Now we can start the proof.
   Let LHS be the left hand side of Equa.(\ref{EquaPropWeakConv}).
   For each $R>0$,
   
   \begin{equation*}
       \begin{aligned}
           \text{LHS}=&
           \frac{1}{C^i_R}
           \int 1_{\wtB_R}([g]_{G_{\hor}^{\circ}}) \psi\circ p ([g]_{\Gamma_U})  \rmm_{[\rmG]_{\Gamma_U}}([g])\\
           =& \frac{1}{C^i_R}
           \int_{\wtB_R}
           \left(
        \psi\circ p (g[h]_{\Gamma_U}) \rmm_{[G_{\hor}^{\circ}]_{\Gamma_U}}([h]) \right) \rmm_{{\rmG}/{G_{\hor}^{\circ}}}([g])\\
           =&\frac{1}{C^i_R} \la \psi,
           \int_{\wtB_R} \left(
           g_* \rmm_{[\rmG_{\hor}^{\circ}]_{\Gamma_i}} 
           \right) \rmm_{\rmG/G_{\hor}^{\circ}}
           \ra \\
           =& \frac{1}{C^i_R} \la \psi,
           \int_{\rmK\times \exp(B_{\fraka}(R))} \left(
           (ka)_* \rmm_{[\rmG_{\hor}^{\circ}]_{\Gamma_i}} 
           \right)  C_7 \widehat{\rmm}_{\rmK}(k)\otimes\mu_A(a)
           \ra \\ 
           =& \frac{1}{C^i_R} \la \psi,
           \int_{\rmK\times \exp(B^{C,+}_{\fraka}(R))} \left(
           (ka)_* \rmm_{[\rmG_{\hor}^{\circ}]_{\Gamma_i}} 
           \right)  C_7 \widehat{\rmm}_{\rmK}(k)\otimes\mu_A(a)
           \ra.
       \end{aligned}
   \end{equation*}
   Lemma \ref{lemDiverg} is applied in the last step.
   
   Now take the limit. We will treat the case of $\limsup$ only. The other one $\liminf$ is similar.
   By applying Lemma \ref{lemAsymVolCompare},
   \begin{equation*}
   \begin{aligned}
           &\limsup_{R\to \infty} \frac{1}{C^i_R} \la \psi,
           \int_{\rmK\times \exp(B^{C,+}_{\fraka}(R))} \left(
           (ka)_* \rmm_{[\rmG_{\hor}^{\circ}]_{\Gamma_i}} 
           \right)  C_7 \widehat{\rmm}_{\rmK}(k)\otimes\mu_A(a)
           \ra \\
           \leq &
           \limsup_{R\to \infty} \frac{1}{C^i_R} \la \psi,
           \int_{\rmK\times \exp(B^{+}_{\fraka}(R,\delta))} \left(
           (ka)_* \rmm_{[\rmG_{\hor}^{\circ}]_{\Gamma_i}} 
           \right)  C_7 \widehat{\rmm}_{\rmK}(k)\otimes\mu_A(a)
           \ra 
           + \ep\\
           \leq &
           \limsup_{R\to \infty} \la \psi,
            \frac{\Vol(\rmG/\Gamma_i)}{ \mu_{A}( B^+_{\fraka}(R,\delta)) }
           \int_{\rmK\times \exp(B^{+}_{\fraka}(R,\delta))} \left(
           (ka)_* \widehat{\rmm}_{[\rmG_{\hor}^{\circ}]_{\Gamma_i}} 
           \right)  \widehat{\rmm}_{\rmK}(k)\otimes\mu_A(a)
           \ra 
           + \ep\\
           =& \la \psi , {\rmm}_{\rmG/\Gamma_i} \ra + \ep.
   \end{aligned}
   \end{equation*}
   Now the proof is complete.

\subsection{Decomposition of Haar measures}

In this section we collect some more-or-less standard facts on expressions of Haar measures in various coordinates. Some arguments are provided when we fail to identify a precise reference. General references include
\cite[Ch.I, Sec.5]{Helga2nd} and \cite[Ch.VIII]{Knap02}

 Recall that the measure $\rmm_{H}$ for a closed subgroup $H$ of $\rmG$ refers to the measure induced from the induced Riemannian metric, which is induced from the trace form. And for a finite measure $\mu$, define a probability measure $\whmu:=\mu/|\mu|$.

We have the bijection
\begin{equation*}
    \Phi_4 : (\rmK \rmM_{\scrI_0}) \times \rmA_{\scrI_0} \times \rmU_{\scrI_0} \cong \rmG,
\end{equation*}
and the surjection
\begin{equation*}
    \Phi_5: \rmK \times  \rmM_{\scrI_0} \times \rmA_{\scrI_0} \times \rmU_{\scrI_0}
    \twoheadrightarrow
     (\rmK \rmM_{\scrI_0}) \times \rmA_{\scrI_0}\times \rmU_{\scrI_0}.
\end{equation*}

By \cite[Proposition 8.44]{Knap02}, there exists a constant $C_4>0$
 such that 
 \begin{equation}\label{EquaConstLanglDecomp}
     (\Phi_4\circ \Phi_5)_*
 \left( C_4\cdot \widehat{\rmm}_{\rmK}\otimes \rmm_{\rmM_{\scrI_0}} \otimes \alpha_{\scrI_0}(a)\rmm_{\rmA_{\scrI_0}}(a) \otimes \rmm_{\rmU_{\scrI_0}}
 \right) = \rmm_{\rmG}.
 \end{equation}
 
 Now we seek to determine the constant $C_4$ (see Lemma \ref{lemConstLanglDecomp} below).
 
 On $\rmK\rmM_{\scrI_0}$, there are two natural measures. One is obtained from the push forward of $\rmm_{\rmK}\otimes \rmm_{\rmM_{\scrI_0}}$ called $\mu$. The other one is to regard $\rmK\rmM_{\scrI_0}$ as a closed submanifold of $\rmG$ and obtain one from the induced Riemannian metric called $\nu$. 
 
 \begin{lem}
 Notations as above, 
 \begin{equation}\label{equaCartanCompFibre}
     \mu=\Vol(\rmK_{\scrI_0})\nu. 
 \end{equation}
 \end{lem}
 
 Before the proof, recall some general constructions on principal bundles (see \cite[27.1]{TuDiffGeo}).
 Let $\rmK_0$ be a compact connected Lie group equipped with a top-degree nonzero $\rmK_0$-invariant differential form $\omega_{\rmK_0}$. Let 
 \begin{equation*}
     \begin{tikzcd}
            E \arrow[d,"\pi"] \arrow[loop right, "K_0"]\\
            B&
     \end{tikzcd}
 \end{equation*}
 be a principal $K_0$-bundle where $K_0$ acts from the right.
 Take a local trivialization $\{(U_i,\phi_{U_i})\}$,
 \begin{equation*}
 \begin{tikzcd}
        \pi^{-1}(U_i) \arrow[r,"\cong"]\arrow[r,swap,"\phi_{U_i}"]
        \arrow[d,"\pi"]
        &
        U_i \times K_0 \arrow[dl,"p_i"] \\
        U_i.
 \end{tikzcd}
 \end{equation*}
 For an index $i$ and $u\in U_i$, define $\iota_{i,u}(k):= (u,k)$ from $K_0$ to $U_i \times K_0$. And let $q_i: U_i \times K_0 \to K_0$ be the natural projection. Thus the above diagram can be completed as 
  \begin{equation*}
 \begin{tikzcd}
        \pi^{-1}(U_i) \arrow[r,"\cong"]\arrow[r,swap,"\phi_{U_i}"]
        \arrow[d,"\pi"]
        &
        U_i \times K_0 \arrow[dl,"p_i"] \arrow[d, "q_i"]
        & K_0 \arrow[l,swap, "\iota_{i,u}"]
        \\
        U_i & K_0 &.
 \end{tikzcd}
 \end{equation*}
 
 Note that transition maps between different charts take the form
 \begin{equation*}
     \begin{tikzcd}
         &\pi_{-1} ( U_i\cap U_j) \arrow[dl,swap,"\phi_{U_i}"]\arrow[dr,"\phi_{U_j}"]&
        \\
            ( U_i\cap U_j)\times K_0 \arrow[rr,"{(u,k)\mapsto(u,k_{ji}(u)k)}"]
            \arrow[rr,swap, "\phi_{ji}"]
            & &( U_i\cap U_j)\times K_0  \\
     \end{tikzcd}
 \end{equation*}
 for some $k_{ji}(u)\in K_0$ depending on $i,j,u$.
 
 Define a (in general incompatible) system of degree $k_0:=\dim K_0$ differential forms $(\pi^{-1}(U_i), \omega_i^{\perp})$ by
 \begin{equation*}
     \omega_i^{\perp}:= \phi_{U_i}^*(q_i^*\omega_{K_0}).
 \end{equation*}
 One can check that (for every $i$ and $u\in U_i$)
 \begin{equation}\label{equaPBdleFibreFormPullBack}
     \iota_{i,u}^* (q_i^*\omega_{K_0})=\omega_{K_0}.
 \end{equation}
 
 Let $\omega_B$ be a top-degree differential form on the base $B$. One can check that the system of differential forms
 \begin{equation*}
     (\pi^{-1}(U_i), \pi^* \omega_B \wedge
     \omega_i^{\perp})
 \end{equation*}
 is compatible and hence glue together to a global top-degree differential form $\wtomega_{B}$ on $E$. Moreover, for any other system of differential forms $(\pi^{-1}(U_i),\phi_{U_i}^*\eta_i)$, if 
 $$(\pi^{-1}(U_i),\pi^* \omega_B \wedge \phi_{U_i}^*\eta_i)$$ is compatible and if $\iota_{i,u}^* (\eta_i)=\omega_{K_0}$ for all indices $i$ and $u\in U_i$, then for all $i$,
 \begin{equation}\label{equaUniqLiftForm}
     \pi^* \omega_B \wedge
     \omega_i^{\perp} 
     = \pi^* \omega_B \wedge
     \phi_{U_i}^*\eta_i.
 \end{equation}
 Let $\omega'_B$ be a top-degree differential form of compact support on $B$ and define $\wtomega'_B$ in the same way as $\wtomega_B$ above.
 Using local coordinates and partition of unity, one can check that
 \begin{equation}\label{equaPrinBdleInteg}
     \int_E \wtomega'_B = \int_{B} \omega'_B \cdot \int_{K_0} \omega_{K_0}.
 \end{equation}
 
 Now suppose that both the total space and the base admit an action (by diffeomorphisms) of a(n abstract) group $H$ commuting with the action of $K_0$ and that $\pi$ is $H$-equivariant.
 Then one can check, using Equa.(\ref{equaUniqLiftForm}), that
 \begin{equation*}
     \omega_B \text{ is } H\text{-invariant } \implies 
     \wtomega_B \text{ is } H\text{-invariant}.
 \end{equation*}
 
 \begin{proof}
  Take $B:=\rmK \rmM_{\scrI_0} $, $E:= \rmK\times \rmM_{\scrI_0} $ and $\pi: E\to B$ be the group multiplication. Let $K_0:=\rmK_{\scrI_0}$ acting (from right) on $E$ by
  \begin{equation*}
      R_{x}\cdot (k,m) := (kx,x^{-1}m).
  \end{equation*}
  Then this gives a (trivial) principal $\rmK_{\scrI_0}$-bundle.
  Indeed, one has the following commutative diagrams
 \begin{equation*}
 \begin{tikzcd}
        \rmK\times \rmM_{\scrI_0} \arrow[r,swap,"\cong"] \arrow[r,"f_1"]
        \arrow[d,"\pi"]
        & \rmK \times \rmK_{\scrI_0} \times \exp(\frakp\cap \frakm_{\scrI_0})\arrow[d] &\\
       \rmK \rmM_{\scrI_0} \arrow[r,swap, "\cong"]
       \arrow[r,"f_2"]
       & \rmK \times \exp(\frakp\cap \frakm_{\scrI_0})
 \end{tikzcd}
 \end{equation*}
where $\frakp$ is the $(-1)$-eigenspace in $\frakg$ with respect to the Cartan involution associated with $\rmK$.
Also, in the above diagram,
the right vertical arrow is the natural projection, the left vertical arrow  and $f_2^{-1}$ are just multiplication, and
$f_1^{-1}$ is defined by $(k,x,m)\mapsto (kx,x^{-1}m)$.

Let $H:=\rmK \times \rmM_{\scrI_0}$ act from left by
\begin{equation*}
    (a,b)\cdot (k,m):= (ak,mb^{-1}).
\end{equation*}

Now we take $\omega_B$ (resp. $\omega_E$) be the $H$-invariant volume form on $B$ (resp. $E$) inducing the measure $\nu$ (resp. $\mu$).  To prove the lemma, it suffice to show that for a compactly supported smooth function $f$ on $B$,
\begin{equation*}
    \int_X (\pi^*f) \omega_E = \Vol(\rmK_{\scrI_0})\cdot \int_B f \omega_B .
\end{equation*}

By discussion above, $\wtomega_B$ is $H$-invariant and hence there exists $C>0$ such that 
\begin{equation*}
    \wtomega_B = \pm C \cdot \omega_E.
\end{equation*}
To identify $C$, it suffices to compare their values at the identity.

Let $l_0:=\dim \rmK_{\scrI_0}$, $l_1:=\dim \rmK$ and 
$l_2:= \dim \frakp \cap \frakm_{\scrI_0}$.
Let 
\begin{equation*}
\begin{aligned}
     &\ONB_{\frakk_{\scrI_0}}:=
    \left\{
    v_1,...,v_{l_0}
    \right\}   ,\quad
    \ONB_{\frakk}:=
    \left\{
    v_1,...,v_{l_0},...,v_{l_1}
    \right\}, \\
    &  \ONB_{\frakp \cap \frakm_{\scrI_0}}
    := \left\{
    w_1,...,w_{l_2}
    \right\}
\end{aligned}
\end{equation*}
be ortho-normal bases for $\frakk_{\scrI_0}$, $\frakk$ and $\frakp \cap \frakm_{\scrI_0}$ respectively. Then it follows that 
\begin{equation*}
\begin{aligned}
    \scrB_1:= \Big(&
    \sqrt{2}^{-1}(v_1,-v_1),...,\sqrt{2}^{-1}(v_{l_0},-v_{l_0}),
    \sqrt{2}^{-1}(v_1,v_1),...,\\
    &\sqrt{2}^{-1}(v_{l_0},v_{l_0}),(v_{l_0+1},0)...,(v_{l_1},0),(0,w_1),...,(0,w_{l_2})
    \Big)
\end{aligned}
\end{equation*}
is a set of ortho-normal basis for $\frakk\oplus \frakm_{\scrI_0}$
and their projections to $\frakk + \frakm_{\scrI_0}\subset \frakg$ under $\pi_*$ become
\begin{equation*}
\begin{aligned}
    \scrB_2:= \big(
    0,...,0,
    \sqrt{2}v_1,...,
    \sqrt{2}v_{l_0},v_{l_0+1},...,v_{l_1},w_1,...,w_{l_2}
    \big).
\end{aligned}
\end{equation*}
For $i=1,2$, let $\theta_i$ be obtained by wedging nonzero vectors from $\ONB_i$ together. 
Let
\begin{equation*}
    \theta_3:=\frac{1}{\sqrt{2}}v_1\wedge...\wedge \frac{1}{\sqrt{2}}v_{l_0} \in \wedge^{l_0}\frakk_{\scrI_0},
\end{equation*}
whose push-forward to $\frakk\oplus\frakm$ under the differential at identity of $x\mapsto (x,x^{-1})$ gives the first $l_0$ vectors in $\scrB_1$.

By definition $\la \omega_E , \theta_1 \ra=\pm1$.
By tracing the definition of $\wtomega_B$ and using 
Equa.(\ref{equaPBdleFibreFormPullBack}), 
one shows that 
\begin{equation*}
    \la \wtomega_B, \theta_1 \ra= \pm
    \la \omega_B ,\theta_2 \ra \cdot 
    \la \omega_{\rmK_{\scrI_0}} ,  \theta_3  \ra
    = \pm1.
\end{equation*}
Thus $\wtomega_B =\pm \omega_E$.

Therefore by Equa.(\ref{equaPrinBdleInteg})
\begin{equation*}
    \int_E \pi^*f \,\omega_E 
    = \pm \int_E \pi^*f \,\wtomega_B = \pm \int_B f \omega_B \cdot  \Vol(\rmK_{\scrI_0})
\end{equation*}
and we are done.
 \end{proof}

These constructions also yield a formula for $\Vol(\SO_n(\R))$,
making corollary \ref{coroCountSLZ} more explicit, which should be well-known. A proof is provided for the sake of completeness.
\begin{lem}\label{lemVolCptLieGp}
Let $\Vol(\SO_n(\R))$ be induced from the trace form. Then 
\begin{equation*}
    \Vol(\SO_n(\R)) = \Vol(\SO_{n-1}(\R))\cdot \Vol(S^{n-1})\cdot 2^{\frac{n-1}{2}}
\end{equation*}
where $S^{n-1}$ is the standard unit sphere on $\R^n$ and $\Vol(S^{n-1})$ is taken with respect to the standard Euclidean metric on $\R^n$. Combining with the formula for area of the unit sphere we get
\begin{equation*}
    \Vol(\SO_n(\R))= 2^{\frac{n(n-1)}{4}} \prod_{k=2}^{n}\Vol(S^{k-1})
    =2^{\frac{n(n-1)}{4}}
    \prod_{k=2}^{n} \frac{2\pi^{k/2}}{\Gamma(\frac{k}{2})}.
\end{equation*}
\end{lem}

\begin{proof}
    It suffices to prove the first recursive formula.
    For simplicity we write $\rmK_n:=\SO_n(\R)$ and $\frakk_n$ for its Lie algebra. Embed $\rmK_{n-1}$ in $\rmK_n$ at the lower right block. 
    Let $(e_1,...,e_n)$ be the standard basis for $\R^n$.
    Then $\rmK_{n-1}$ is precisely the stabilizer of $e_1$ in $\rmK_n$.
    Then 
    \begin{equation*}
        \ONB_n:=\left\{
        \sqrt{2}^{-1}(E_{ij}-E_{ji})\;\middle\vert\;1\leq i<j\leq n
        \right\}
    \end{equation*}
    forms an ortho-normal basis of $\frakk_n$.
    Identify the tangent space of $\rmK_n/\rmK_{n-1}$ at the identity coset as $\frakk_{n-1}^{\perp}$, the $\R$-subspace of $\frakk_n$ spanned by
    \begin{equation*}
         \ONB_{n-1}^{\perp}:=\left\{
        \sqrt{2}^{-1}(E_{1j}-E_{j1})
        \;\middle\vert\;
        j=2,...,n
        \right\}
    \end{equation*}
    Let $\omega_{\rmK_n/\rmK_{n-1}}$ be the invariant volume form on $\rmK_n/\rmK_{n-1}$ whose value at identity coset is given by the dual of $\ONB_{n-1}^{\perp}$.
    Let $E:=\rmK_n$, $B=\rmK_n/\rmK_{n-1}$ and $\pi:E\to B$ be the natural projection. Then they give a $\rmK_{n-1}$-principal bundle over $B$. Also let $H:=\rmK_n$ naturally act from the left.
    Let $\omega_{\rmK_n}$ be the volume form deduced from the trace form. So we have a lift $\wtomega_{\rmK_n/\rmK_{n-1}}$ on $E$ and
    \begin{equation*}
        \int_E \wtomega_{\rmK_n/\rmK_{n-1}} = \pm
        \int_B \omega_{\rmK_n/\rmK_{n-1}} \cdot \Vol(\rmK_{n-1}).
    \end{equation*}
    Since $\wtomega_{\rmK_n/\rmK_{n-1}}$ is $H$-invariant,
    a computation at the tangent space at identity shows that 
    \begin{equation*}
        \wtomega_{\rmK_n/\rmK_{n-1}}= \omega_{\rmK_n}.
    \end{equation*}
    Thus 
    \begin{equation*}
        \Vol(\rmK_n)= \pm \Vol(\rmK_{n-1}) \cdot  \int_B \omega_{\rmK_n/\rmK_{n-1}}.
    \end{equation*}
    Consider the diffeomorphism $\varphi: [k]_{\rmK_{n-1}}\to k\cdot e_1$ from $\rmK_n/\rmK_{n-1} \to S^{n-1}$.
    Then the measure induced from $\omega_{\rmK_n/\rmK_{n-1}}$ is pushed to a $\rmK_n$-invariant measure on $S^{n-1}$ and thus is proportional to the $\Vol$ measure on $S^{n-1}$.
    To see the proportional scalar, we compute the differential at the identity coset:
    \begin{equation*}
        \diff\varphi_{[id]_{\rmK_{n-2}}}(E_{1j}-E_{j1}) = (E_{1j}-E_{j1})\cdot e_1 = -e_j.
    \end{equation*}
    Thus the Jacobian is ${2}^{-(n-1)/2}$ and 
    \begin{equation*}
        \int_B \omega_{\rmK_n/\rmK_{n-1}} ={2}^{\frac{n-1}{2}} \Vol(S^{n-1}).
    \end{equation*}
    So we are done.
\end{proof}

\begin{lem}\label{lemConstLanglDecomp}
Let $C_4$ be as in Equa.(\ref{EquaConstLanglDecomp}), then
\begin{equation*}
    C_4=\frac{\Vol(\rmK)}{\Vol(\rmK_{\scrI_0})} \cdot {2}^{\frac{\sum_{i<j} |I_i||I_j|}{2}} .
\end{equation*}
\end{lem}

\begin{proof}
By Equa.(\ref{equaCartanCompFibre}) above, 
\begin{equation*}
\begin{aligned}
        &(\Phi_5)_* (
    \frac{\Vol(\rmK)}{\Vol(\rmK_{\scrI_0})} \widehat{\rmm}_{\rmK}\otimes\rmm_{\rmM_{\scrI_0}}\otimes\alpha_{\scrI_0}(a)\rmm_{\rmA_{\scrI_0}}(a)\otimes\rmm_{\rmU_{\scrI_0}} 
    )\\
    &=\nu \otimes\rmm_{\rmM_{\scrI_0}}\otimes\alpha_{\scrI_0}(a)\rmm_{\rmA_{\scrI_0}}(a)\otimes\rmm_{\rmU_{\scrI_0}} .
\end{aligned}
\end{equation*}
It remains to consider the constant brought by $(\Phi_4)_*$, which is a computation of the Jacobian of the differential at the identity.

    Let $E_{i,j}$ (or $E_{ij}$ if no confusion might arise) be the matrix whose $(i,j)$-th entry is $1$ and is equal to $0 $ elsewhere.
    Then 
    \begin{equation*}
    \Tr(E_{i,j}E_{s,t})=
            \begin{cases}
     1 \quad i=t, \,j=s;\\
     0 \quad \text{otherwise}.
    \end{cases}
    \end{equation*}
    Thus 
    \begin{equation*}
        \ONB_{\rmU_{\scrI_0}}:=\left\{
        E_{i,j}\;\vert\;
        i<j,\,i \not\sim_{\scrI_0} j
        \right\}
    \end{equation*}
    forms an ortho-normal basis (to be abbreviated as ONB below) of $\fraku_{\scrI_0}$.
    Also,
    \begin{equation*}
        \calW:=\left\{
        \frac{1}{\sqrt{2}}(E_{i,j}+E_{j,i})\;\middle\vert\;
        i<j,\,i \not\sim_{\scrI_0} j
        \right\}
    \end{equation*}
    is a set of ONB for the subspace spanned by it.
    Fix $\ONB_{\scrI_0}:=\{v_1,...,v_l\}$, some set of ONB for $\frakk+\frakm_{\scrI_0}+\fraka_{\scrI_0}$. Then 
    $ \ONB_{\scrI_0}\sqcup \calW$ is a set of ONB for $\frakg$
    and $\ONB_{\scrI_0}\sqcup \ONB_{\rmU_{\scrI_0}}$ is a ONB for $(\frakk+\frakm_{\scrI_0})\oplus\fraka_{\scrI_0}\oplus \fraku_{\scrI_0}$.
    
    By comparing $\ONB_{\scrI_0}\sqcup \ONB_{\rmU_{\scrI_0}}$
    against 
    $ \ONB_{\scrI_0}\sqcup \calW$  under $\diff\Phi_4\vert_{id}$ and noting that 
    \begin{equation*}
        \begin{tikzcd}
               E_{ij}\arrow[r,"\diff\Phi_4\vert_{id}"]&
               \sqrt{2}^{-1} \frac{E_{ij}+E_{ji}}{\sqrt{2}}+\frac{E_{ii}-E_{ji}}{2} 
               \in \sqrt{2}^{-1} \frac{E_{ij}+E_{ji}}{\sqrt{2}}+ \frakk,
        \end{tikzcd}
    \end{equation*}
    we get
    \begin{equation*}
        |\det (\diff\Phi_4\vert_{id})| = {2}^{-\frac{\sum_{i<j} |I_i||I_j|}{2}} .
    \end{equation*}
    Therefore
    \begin{equation*}
        (\Phi_4)_*\left(
        {2}^{-\frac{\sum_{i<j} |I_i||I_j|}{2}} \cdot 
        \nu
        \otimes \alpha_{\scrI_0}(a)\rmm_{\rmA_{\scrI_0}}(a)
        \otimes \rmm_{\rmU_{\scrI_0}}
        \right)
        =\rmm_{\rmG}.
    \end{equation*}
    And we are done.
\end{proof}

Next we turn to the Cartan decomposition, by \cite[Theorem 5.8]{Helga2nd}, 
the map
\begin{equation*}
    \Phi_6: \rmK_{\scrI_0} \times \rmA_{\rmM_{\scrI_0}}^+ \times \rmK_{\scrI_0} \to \rmM_{\scrI_0}
\end{equation*}
induces 
\begin{equation*}
    (\Phi_6)_*(
    C_6
    \widehat{\rmm}_{\rmK_{\scrI_0}} \otimes
    \prod_{i<j,i\sim_{\scrI_0} j} \frac{\alpha_{ij}(a)-\alpha_{ji}(a)}{2}
    \rmm_{\rmA^{+}_{\rmM_{\scrI_0}}}(a) \otimes \widehat{\rmm}_{\rmK_{\scrI_0}}
    ) = \rmm_{\rmM_{\scrI_0}}
\end{equation*}
for some constant $C_6>0$.

\begin{lem}\label{lemConstKAKdecomp}
We have that
\begin{equation*}
    C_6= \Vol(\rmK_{\scrI_0})^2.
\end{equation*}
\end{lem}

\begin{proof}

It suffices to do this for $\Phi'_6:\SO_n(\R)\times \rmA^+ \times \SO_n(\R) \to \SL_n(\R)$ where $\rmA^+$ is the subgroup of $\SL_n(\R)$ defined by
\begin{equation*}
    \left\{
    a=\diag(a_1,...,a_n)\;\middle\vert\;
    \alpha_{ij}(a)=a_i/a_j \geq 1,\; \forall i<j;\;a_i>0,\,\forall i
    \right\}.
\end{equation*}
Take $a_0$ in the interior of $\rmA^+$. 
It suffices to show that the determinant of the differential of 
\begin{equation*}
    \Phi: (k_1,a,k_2) \mapsto k_1ak_2a_0^{-1}
\end{equation*}
at $(id, a_0,id)$ from $\frakk_n \oplus \fraka \oplus \frakk_n$ to $\fraksl_n$ is equal to $\pm1$. Here $\frakk_n$ denotes the Lie algebra of $\SO_n(\R)$.

From the definition we see that 
\begin{equation*}
    \begin{aligned}
            &\diff \Phi\vert_{(id,a_0,id)} (v_1,v_2,0) =v_1+v_2;\\
            &\diff \Phi\vert_{(id,a_0,id)} (0,0,w) =\Ad(a_0)w.
    \end{aligned}
\end{equation*}

Let 
\begin{equation*}
     \ONB_{\frakk_n}:=\left\{
     \frac{1}{\sqrt{2}}( E_{ij}-E_{ji}) \;\middle\vert\;
     i<j
    \right\}
\end{equation*}
be an ortho-normal basis on $\frakk_n$. 
Then for $i<j$,
\begin{equation*}
    \diff \Phi\vert_{(id,a_0,id)} (0,0,\frac{1}{\sqrt{2}}( E_{ij}-E_{ji})) \in 
    \left(
    \frac{\alpha_{ij}(a_0)-\alpha_{ji}(a_0)}{2} \cdot 
    \frac{1}{\sqrt{2}}(E_{ij}+E_{ji})
    \right)+ \frakk_n.
\end{equation*}
By taking the wedge of them, we are done.
\end{proof}

Combining Lemma \ref{lemConstLanglDecomp} and \ref{lemConstKAKdecomp}, we have shown
\begin{lem}\label{lemHaarDecompC7}
Let $C_7$ be as in Lemma \ref{lemDecompHaar}, then
\begin{equation*}
    C_{7}= {\Vol(\rmK)} \cdot 2^{-\frac{\sum_{i<j}|I_i||I_j|}{2}}.
\end{equation*}
The natural map defined by taking multiplication
\begin{equation*}
    \Phi_7: \rmK \times \rmA^+_{\rmM_{\scrI_0}} \times \rmA_{\scrI_0} \times \rmK_{\scrI_0} \times \rmU_{\scrU_{\scrI_0}} \to \rmG
\end{equation*}
induces
\begin{equation*}
    (\Phi_7)_*
    \left( C_7 \cdot 
    \rho_{\scrI_0}(a,b)
     \cdot 
     \widehat{\rmm}_{\rmK} \otimes 
     \rmm_{\rmA^{+}_{\rmM_{\scrI_0}}}(a) \otimes
      \rmm_{A_{\scrI_0}}(b)\otimes
      {\rmm}_{\rmK_{\scrI_0}}
      \otimes \rmm_{\rmU_{\scrI_0}}
    \right) = \rmm_{\rmG}.
\end{equation*}
where we have used the shorthand
\begin{equation*}
    \rho_{\scrI_0}(a,b):= \alpha_{\scrI_0}(b)\cdot
     \prod_{i<j,i\sim_{\scrI_0} j} \frac{\alpha_{ij}(a)-\alpha_{ji}(a)}{2}.
\end{equation*}
\end{lem}

\subsection{Proof of the lemmas}
   
   \subsubsection{Proof of Lemma \ref{lemGhoroDescription}}\label{prooflemGhoro}
  
   \begin{proof}
   Compare \cite[Remark 2.28]{KapoLeebPor17}.
    So take $g\in\rmG$ with $g\rmU_{\scrI_0}\rmK=\rmU_{\scrI_0}\rmK$.
    In particular $g\in \rmU_{\scrI_0}\rmK$, so there exists $u_g \in \rmU_{\scrI_0}$, $k_g\in\rmK$ such that $g=u_gk_g$. 
   
    Then 
    \begin{equation*}
        u_gk_g \rmU_{\scrI_0} \subset \rmU_{\scrI_0}\rmK
        \implies k_g \rmU_{\scrI_0} k_g^{-1}\subset \rmU_{\scrI_0}\rmK
        \implies k_g \rmU_{\scrI_0} k_g^{-1}\subset \rmK\rmU_{\scrI_0}.
    \end{equation*}
    Take a representation $(\rho,V)$ and a vector $v\in V$ such that $\bmU_{\scrI_0}$ is equal to the stabilizer of $v$. Hence
    \begin{equation*}
        k_g \rmU_{\scrI_0} k_g^{-1} \cdot v \subset \rmK \rmU_{\scrI_0}\cdot v
    \end{equation*}
    is bounded. By the feature of  polynomials, we must have
    \begin{equation*}
         k_g \rmU_{\scrI_0} k_g^{-1} \cdot v =v.
    \end{equation*}
    In particular $k_g$ normalize $\bmU_{\scrI_0}$. 
    On the other hand, it is direct to check that $N_{\rmK}(\rmU_{\scrI_0})\rmU_{\scrI_0}\subset G_{\hor}$.
    So we are done.
   \end{proof}
   
   \subsubsection{Proof of Lemma \ref{lemGammaGhoro}}\label{proofGammaGhoro}
   \begin{proof}
   The nontrivial direction is $ G_{\hor} \cap \Gamma_i  \subset \rmU_{\scrI_0} \cap \Gamma_i $.
   So take $\gamma \in  G_{\hor}\cap \Gamma_i$. By Lemma \ref{lemGhoroDescription}, 
   $G_{\hor}= N_{\rmK}(\rmU_{\scrI_0})\rmU_{\scrI_0}$ and we may write 
   $\gamma= k_{\gamma} u_{\gamma}$.
   As in the proof of last lemma, take a $\Q$-representation $(\rho,V)$ and a vector $v\in V(\Q)$ such that $\bmU_{\scrI_0}$ is equal to the stabilizer of $v$. Consider $(\gamma^n\cdot v)$. On the one hand, this is a discrete set. On the other hand
   \begin{equation*}
       \gamma^n\cdot v= k_{\gamma}^n u_n' \cdot v= k_{\gamma}^n \cdot v
   \end{equation*}
   for some $u_n' \in \rmU_{\scrI_0}$. And hence $(\gamma^n\cdot v)$ is bounded. Thus there exists $m<n$ such that 
   \begin{equation*}
       \gamma^n \cdot v =\gamma^m \cdot v \implies \gamma^{n-m}\cdot v= v.
   \end{equation*}
   Hence $\gamma^{n-m}$ is a unipotent matrix in $\rmU_{\scrI_0}$. As $\Gamma_i$ is assumed to be neat, we conclude that $\gamma$ is a unipotent matrix. Hence by considering its logarithm, we see that $\gamma\in \rmU_{\scrI_0}$.
    \end{proof}
    
    \subsubsection{Proof of Lemma \ref{lemDecompHaar}}\label{proofDecompHaar}
    \begin{proof}
    By definition, $\rmm_{\rmG/\rmG_{\hor}^{\circ}}$ is the unique $\rmG$-invariant measure on $\rmG/\rmG_{\hor}^{\circ}$ such that the fibre integration formula
    \begin{equation*}
        \int_{\rmG/\rmG_{\hor}^{\circ}} 
        \int_{\rmG_{\hor}^{\circ}}
        f(gh)  \rmm_{\rmG_{\hor}^{\circ}}(h)
        \rmm_{\rmG/\rmG_{\hor}^{\circ}}([g])
        =\int_{\rmG} f(x) \rmm_{\rmG}(x)
    \end{equation*}
    holds for all compactly supported function $f$ on $\rmG$.
    Thus we are done by Lemma \ref{lemHaarDecompC7}.
    \end{proof}
    
    \subsubsection{Proof of Lemma \ref{lemHtBallDecomp}}\label{proofHtBallDecomp}
    \begin{proof}
    This follows in the same way as in \cite[Lemma 29]{GolMoha14}.
    Indeed, \cite[Theorem 3.6]{Busemann48} implies that for $m \in \rmM_{\scrI_0}\rmA_{\scrI_0}$ and $u \in \rmU_{\scrI_0}$, $\dist([mu],[id])\geq \dist([m],[id])$.
    
    \end{proof}
    
    \subsubsection{Proof of Lemma \ref{lemVolumeComp}}\label{proofVolumeComp}
    \begin{proof}
     This follows by applying \cite[Lemma 25,26]{GolMoha14}, which builds on \cite{GoroWeiss07}. 
     So we need to feed it with an Euclidean space, a vector $\bmv_0$, a cone $\calC$ with $\bmv_0 \in \calC$.
     We identify $\fraka$, the Lie algebra of the full diagonal group in $\rmG$, with $\R^{N-1}$ such that the trace form goes to the standard Euclidean metric.
     Note that $\fraka_{\rmM_{\scrI_0}}\oplus \fraka_{\scrI_0}$ is naturally identified with $\fraka$.
      Recall
     \begin{equation*}
         \calC_{C}:=
         \left\{\bma,\;
         a_i -a_j \geq \max\{0,C\},\;\forall k,\, i<j \in I_k;\;
         \sum_{i\in I_1\cup ...\cup I_k} a_{i} \geq C, \;\forall 1\leq k<k_0
         \right\}.
     \end{equation*}
     where $\bma=\diag(a_1,...,a_{N})$ for an element $\bma\in \fraka^+_{\rmM_{\scrI_0}} \oplus \fraka_{\scrI_0}$.
     Thus $\calC_{0}$ is a cone and $\calC_{C}=\calC_0+x$ for some $x\in \fraka$.
     Let 
     \begin{equation*}
         \bmv_0:= \sum_{1\leq i<j\leq N} E_{ii}-E_{jj}
         =\diag(
         N-1,N-3,...,-N+1
         )
     \end{equation*}
     where $E_{ii}$ is the matrix whose $(i,i)$-th entry is $1$ and is $0$ elsewhere. One can also check that $\bmv_0 \in \calC_0$.
     Then 
     \cite[Lemma 25,26]{GolMoha14} implies that  for every $C\in \R$,
     \begin{equation*}
         \int_{\norm{\bmy}\leq R, \bmy\in \calC_C} e^{\Tr(\bmv_0 \bmy)} \diffbmy \sim 
         \left(
         \frac{2\pi R}{\norm{\bmv_0}}\right)
         ^{\frac{N-2}{2}}\cdot e^{\norm{\bmv_0}R}  ,
     \end{equation*}
     indenpendent of $C$.
     Now it suffices to show that 
     \begin{equation}\label{equaGMlemma25}
         \mu_A(B^+_{\fraka}(R))
         \sim
         \left(\frac{1}{2}\right)^{\sum_{k=1}^{k_0}\frac{|I_k|(|I_k|-1)}{2}}
         \int_{\norm{\bmy}\leq R, \bmy\in \calC_C} e^{\Tr(\bmv_0 \bmy)} \diffbmy.
     \end{equation}
     By unfolding the definitions, 
     \begin{equation*}
     \begin{aligned}
         &\left(\frac{1}{2}\right)^{\sum_{k=1}^{k_0}\frac{|I_k|(|I_k|-1)}{2}} \int_{\norm{\bmy}\leq R, \bmy\in \calC_C} e^{\Tr(\bmv_0 \bmy)} \diffbmy \\
         = &
         \int_{\norm{\bmy}\leq R,\bmy \in \calC_C} \exp\left({\sum_{1\leq s<t\leq k_0} \sum_{i\in I_s,j\in I_t} y_i-y_j
         }\right)
         \cdot 
         \prod_{k=1}^{k_0} \prod_{i<j \in I_k} \frac{1}{2} e^{y_i-y_j}
         \diffbmy.
     \end{aligned}
     \end{equation*}
     By comparison, from the definition of $\mu_A$, we have
     \begin{equation*}
     \begin{aligned}
       \mu_A(B^+_{\fraka}(R)) = 
         \int_{\norm{\bmy}\leq R,\bmy \in \calC_0} &\exp\left({\sum_{1\leq s<t\leq k_0} \sum_{i\in I_s,j\in I_t} y_i-y_j
         }\right)\\
         &\cdot 
         \prod_{k=1}^{k_0} \prod_{i<j \in I_k} \frac{1}{2}e^{y_i-y_j}
         \left(
          1-e^{-2(y_i-y_j)}
         \right)
         \diffbmy
     \end{aligned}
     \end{equation*}
     where $\bmy=\diag(y_1,...,y_{N})$.
     Thus it suffices to show that the contribution of the factors
     \begin{equation*}
          1-e^{-2(y_i-y_j)}
     \end{equation*}
     are negligible. First note that on $\calC_0$, 
     $ 0\leq  1-e^{-2(y_i-y_j)} \leq 1$, thus the LHS of Equa.(\ref{equaGMlemma25}) is no larger than the RHS.

     Recall  $B^{C,+}_{\fraka}(R)=B_{\fraka}(R)\cap \calC_C$.
     For every $C, R>0$,
     \begin{equation*}
            \mu_A(B^{+}_{\fraka}(R))\geq \mu_A(B^{C,+}_{\fraka}(R) ).
     \end{equation*}
      For any $\ep\in(0,1)$, there exists $C_{\ep}>0$ such that 
      \begin{equation*}
          1-e^{-2(y_i-y_j)} \in [1-\ep,1]
      \end{equation*}
      for all $k=1,...,k_0$, $i<j\in I_k$ and $\bmy\in \calC_{C_{\ep}}$.
      Thus 
      \begin{equation*}
          \mu_A(B^{C_{\ep},+}_{\fraka}(R) ) \geq 
          \left((1-\ep)\cdot \frac{1}{2}\right)^{\sum_{k=1}^{k_0}\frac{|I_k|(|I_k|-1)}{2}} \int_{\norm{\bmy}\leq R, \bmy\in \calC_{C_{\ep}}} e^{\Tr(\bmv_0 \bmy)} \diffbmy. 
      \end{equation*}
      Letting $R \to +\infty$ and then $\ep\to 0$ concludes the proof.
    \end{proof}
    
    \subsubsection{Proof of Lemma \ref{lemWellRound}}\label{proofWellRound}
    \begin{proof}
    By an equivalent definition of well-rounded (\cite[Proposition 1.3]{EskMcM93}), we need to show that for every $\ep \in (0,1)$, there exists a neighborhood $\Omega$ of identity in $\rmG$ such that for all $R\geq 1$,
    \begin{equation*}
        (1-\ep)\rmm_{\rmG/G_{\hor}}(\bigcup_{\omega\in\Omega} \omega B_R)
        < \rmm_{\rmG/G_{\hor}}(B_R) <
        (1+\ep)\rmm_{\rmG/G_{\hor}}(\bigcap_{\omega\in\Omega} \omega B_R).
    \end{equation*}
    For every $\delta\in (0,1)$, choose a neighborhood $\Omega_{\delta}$ of identity in $\rmG$  satisfying
    \begin{equation*}
        \dist([\omega^{-1}]_{\rmK}, [id]_{\rmK})\leq \delta,
        \quad \forall \omega \in \Omega_{\delta}.
    \end{equation*}
    From the definition of $B_R$, 
    for all $R\geq 1$ and $\omega \in \Omega_{\delta}$,
    \begin{equation*}
    \begin{aligned}
          &B_{R-\delta} \subset \omega B_R \subset B_{R+\delta}\\
            & \implies 
            \bigcup_{\omega\in\Omega} \omega B_R \subset B_{R+\delta},\;
             \bigcap_{\omega\in\Omega} \omega B_R \supset B_{R-\delta}.
    \end{aligned}
    \end{equation*}
    
     From the description of Haar measures in Lemma \ref{lemVolumeComp},
    for every $\ep\in (0,1)$ we can find $\delta_{\ep}>0$ such that for all $R\geq 1$,
    \begin{equation*}
    \begin{aligned}
        &\rmm_{\rmG/G_{\hor}}(B_{R+\delta_{\ep}}) < \frac{1}{1-\ep} \rmm_{\rmG/G_{\hor}}(B_R) \\
         &   \rmm_{\rmG/G_{\hor}}(B_{R-\delta_{\ep}}) > \frac{1}{1+\ep} \rmm_{\rmG/G_{\hor}}(B_R). 
    \end{aligned}
    \end{equation*}
    Thus by choosing $\Omega:=\Omega_{\delta_{\ep}}$, the well-roundedness is verified.
    \end{proof}
    
    \subsubsection{Proof of Lemma \ref{lemDiverg}}\label{proofDiverg}
    \begin{proof}
    Indeed, 
    for every $\ep>0$ and for $C\in \R$ sufficiently small,
    for every $(k,a,b)\notin \rmK \times \exp(\calC_C)$, one can find an integral vector $v$ in $\wedge^i\Z^N$ for some $i$ such that 
    \begin{equation*}
        \norm{kab\cdot v} \leq \ep.
    \end{equation*}
    The proof is complete by Mahler's criterion. 
    \end{proof}
    
    \subsubsection{Proof of Lemma \ref{lemAsymVolCompare}}\label{proofAsymVolCompare}
    \begin{proof}
        From the definition of $\mu_A$, one sees that there exists a constant $C'>0$ depending on $C$ such that
        \begin{equation*}
            \mu_A\vert_{\calC_C\setminus \calC_0} \leq C'\Leb
        \end{equation*}
        where $\Leb$ is some Lebesgue measure on the linear space $\fraka$.
        Thus 
        \begin{equation*}
            \limsup_{R\to +\infty} \frac{\mu_A(B^{C,+}_{\fraka}(R)\setminus B^+_{\fraka}(R)) 
            }{\Leb(B^+_{\fraka}(R))} <\infty.
        \end{equation*}
        But one also sees from \ref{proofVolumeComp} that 
        \begin{equation*}
            \lim_{R\to+\infty} \frac{\Leb(B^+_{\fraka}(R))}
            {\mu_A(B^+_{\fraka}(R))} =0.
        \end{equation*}
        So this proves the first part.
        The second part follows from Lemma \ref{lemVolumeComp}.
    \end{proof}

\bibliographystyle{amsalpha}
\bibliography{ref}

\end{document}